\RequirePackage{fix-cm}
\documentclass[smallextended,envcountsect]{svjour3}       
\smartqed  

\usepackage{graphicx}
\usepackage{bm}
\usepackage{array}
\usepackage{enumerate}
\usepackage{color}
\usepackage{dsfont}
\usepackage{amsmath}
\usepackage{amssymb}
\usepackage{amsfonts}
\usepackage{calrsfs}
\usepackage{natbib}

\definecolor{brown}{rgb}{0.59, 0.29, 0.0}
\definecolor{BROWN}{rgb}{0.59, 0.29, 0.0}

\renewcommand{\le}{\leqslant}
\renewcommand{\ge}{\geqslant}
\newcommand{\eps}{\varepsilon}

\newcommand{\norm}[1]{\|#1\|}
\newcommand{\E}{\operatorname{E}}

\newcommand{\1}{\mathds{1}}
\newcommand{\dto}{\rightsquigarrow}

\newcommand{\boundary}{\partial}

\newcommand{\defn}{\emph}
\newcommand{\abs}[1]{\lvert#1\rvert}

\newcommand{\Borel}{\mathcal{B}}

\newcommand{\RR}{\mathbb{R}}

\newcommand{\NN}{{\mathbb{N}}}

\newcommand{\law}{{\cal L}}
\newcommand{\Pareto}{\operatorname{Pareto}}
\newcommand{\ZZ}{\mathbb{Z}}
\newcommand{\reals}{\mathbb{R}}

\newcommand{\point}{\,\cdot\,}

\newcommand{\RV}{\mathcal{R}}
\renewcommand{\Pr}{\operatorname{Pr}}
\newcommand{\hypo}{\operatorname{hypo}}
\newcommand{\domain}{\mathbb{D}}
\newcommand{\USC}{\operatorname{USC}}
\usepackage{mathptmx}      
\begin{document}

\title{ Polar decomposition of regularly varying time series in star-shaped metric spaces}

\titlerunning{Regularly varying time series}        

\author{Johan Segers \and Yuwei Zhao \and Thomas Meinguet\thanks{The
    first author's research is supported by contract nr.\ 07/12-045 of
    the Projet d'Actions de Recherche Concert\'ees of the Communaut\'e
    fran\c{c}aise de Belgique, granted by the Acad\'emie universitaire
    Louvain and the second author's research is supported by IAP research network grant nr.\ P7/06 of the Belgian government (Belgian Science Policy).}}

\authorrunning{J. Segers, Y. Zhao and T. Meinguet} 

\institute{Johan Segers \at
              Universit\'{e} catholique de Louvain \\
              ISBA, Voie du Roman Pays 20\\
            B-1348 Louvain-la-Neuve, Belgium\\
\email{johan.segers@uclouvain.be}  
           \and
            Yuwei Zhao \at
Universit\'{e} catholique de Louvain \\
              ISBA, Voie du Roman Pays 20\\
            B-1348 Louvain-la-Neuve, Belgium\\
\email{yuwei.zhao@uclouvain.be}
\and
Thomas Meinguet \at
ING Bank N.V.\\Postbus 1800\\1000 BV Amsterdam\\Netherlands\\
\email{thomas.meinguet@skynet.be}
}

\date{Received: date / Accepted: date}

\maketitle

\begin{abstract}
There exist two ways of defining regular variation of a time series in
a star-shaped metric space: either by the distributions of finite
stretches of the series or by viewing the whole series as a single
random element in a sequence space. The two definitions are shown to
be equivalent. The introduction of a norm-like function, called
\emph{modulus}, yields a polar decomposition similar to
the one in Euclidean spaces. The angular component of the time series, called \emph{angular} or \emph{spectral tail process}, captures all aspects of extremal dependence. The stationarity of the underlying series induces a transformation formula of the spectral tail process under time shifts.
\end{abstract}

\section{Introduction}
The concept of regular variation plays an important role in the study of heavy-tailed phenomena, which appear in diverse contexts such as financial risk management, telecommunications, and meteorology, to name a few. Traditionally, regular variation has been defined and studied for univariate functions and random variables in $\mathbb{R}$, see for instance \citet{BIN87} and \citet{RES87} and the references therein. Later on, it has been extended to random vectors and stochastic processes~\citep{RES86, HL05, RES07}. \cite{BS09} study the polar decomposition of a regular varying time series and \cite{HL06}, by introducing the $M_0$-convergence, build a framework to define regular variation for measures on metric spaces endowed with scalar multiplication. Combining results and methods in these two papers, \cite{SER10} provide a detailed study of regularly varying time series in Banach spaces. Our aim is to extend and generalize results in the latter concerning two aspects: regular variation of the time series when seen as a single random element in a sequence space and the polar decomposition in star-shaped metric spaces.

Let $\bm{X}=(X_t)_{t\in \mathbb{Z}}$ be a discrete-time stochastic process taking values in a star-shaped metric space $S$, i.e., a complete, separable metric space equipped with a scalar multiplication (see Section~\ref{sec:space}). Regular variation of random elements in such spaces has been introduced in \cite{HL06}, generalizing theory in \citet{kuelbs+m:1974} and \citet{mandrekar+z:1980} for regular variation in Hilbert and Banach spaces, respectively; see also \citet{gine+h+v:1990} and~\citet{DHL01} for regular variation of random continuous functions. Regular variation of a time series $\bm{X}$ can be defined via its finite-dimensional distributions, that is, $(X_{-m},\ldots,X_m)$ is regularly varying as a random element in $S^{2m+1}$ for each $m\in \mathbb{Z}_+=\{0,1,2,\ldots\}$. Alternatively, $\bm{X}$ can be required to be regularly varying as a random element in the sequence space $S^{\mathbb{Z}}$. In \cite{Owada:2013vu}, it is shown that, under mild  conditions, these two ways of defining regular variation of a real-valued stochastic process are equivalent. As one of the paper's aims, the equivalence is shown for $\bm{X}$ taking values in a general metric space.  

The polar decomposition of stationary regularly varying time series in Euclidean spaces is introduced by \cite{BS09} and generalized to Banach spaces by \cite{SER10}. Let $\mathbb{B}$ be a Banach space equipped with a norm $\norm{\cdot}$.  Regular variation of a $\mathbb{B}$-valued stationary time series $\bm{X}$ is equivalent to the existence of the limit in distribution of 
\begin{equation*}
  \Big(\norm{X_0}/u, (X_t/\norm{X_0})_{t \in \mathbb{Z}}\Big) \text{ conditionally on } \norm{X_0} > u \text{ as } u \to \infty,
\end{equation*}
where the limit of $\norm{X_0}/u$ given $\norm{X_0} > u$ is assumed to be non-degenerate. 
This leads to a natural decomposition of the limit process into independent modular and angular components. The modular component, the limit in distribution of $\norm{X_0}/u$ given $\norm{X_0} > u$ as $u\to \infty$, is fully determined by the index of regular variation, $\alpha$, of the random variable $\norm{X_0}$, while the angular component, the limit in distribution of $(X_t/\norm{X_0})_{t \in \mathbb{Z}}$ given $\norm{X_0}>u$, captures all aspects of extremal dependence. The angular component is called \defn{spectral tail process}. Stationarity of $\bm{X}$ induces a transformation formula for the spectral tail process under time shifts. The spectral tail process provides a single apparatus to describe a large variety of objects describing extremal dependence: the extremal index~\citep{LB83}, the cluster index \citep{mikosch+w:2014}, the extremogram~\citep{DM09, MZ15}, limits of sums or maxima \citep{BKS12, meinguet:2012}, and Markov tail chains~\citep{Janssen+S:2014, DSM15}.

A general metric space may not possess a norm. However, an alternative
function possessing some key properties of a norm, named \emph{modulus}, might still exist. If this is the case, then the above polar decomposition still goes through, and the time-change formula for the spectral tail process is shown to be preserved.


The structure of the paper is as follows. The conditions on the metric
space and the definition of a modulus function are introduced in
Section~\ref{sec:space}. In Section~\ref{sec:rv}, the polar decomposition of a regularly varying random element in a metric
space is studied. Regular variation of a time series seen as a random element in a sequence space is investigated in Section~\ref{sec:regul-vary-stoch}. Results on the
spectral tail process and on the time-change formula are given in
Sections~\ref{sec:spectral-processes} and~\ref{sec:time-change-formula}, respectively.
Section~\ref{sec:disc} provides some brief discussion in connection to hidden regular variation and Appendix~\ref{sec:m_0-convergence} contains auxiliary results on convergence of measures.

\section{Star-shaped metric spaces}
\label{sec:space}

Let $(S, d)$ be a complete, separable metric space and let $0_S \in S$ be a point in $S$ called `origin'. (To avoid trivialities, assume that $S$ is not equal to $\{0_S\}$.) To define regular variation of measures on the metric space $S$, \cite{HL06} assume that $S$ is equipped with a scalar multiplication. The following is a formal definition of such a multiplication. In the cited paper, conditions (i) and (ii) are not mentioned explicitly.

\begin{definition}
\label{def:mult}
A \emph{scalar multiplication} on $S$ is a map $[0, \infty) \times S \to S : (\lambda, x) \to \lambda x$ satisfying the following properties:
\begin{enumerate}[(i)] 
\item $\lambda_1 (\lambda_2 x) = (\lambda_1 \lambda_2) x$ for all $\lambda_1, \lambda_2 \in [0, \infty)$ and $x \in S$;
\item $1 x = x$ for $x \in S$;
\item the map is continuous with respect to the product topology;
\item if $x \in S_0=S\setminus \{0_S\}$ and if $0 \le \lambda_1 < \lambda_2$, then $d(\lambda_1x, 0_S) < d(\lambda_2x, 0_S)$.
\end{enumerate}
\end{definition}

Let $x \in S_0$. For any
$\lambda \in [0, \infty)$, we have $\lambda (0x) = (\lambda 0) x = 0x$
by (i) in Definition~\ref{def:mult}. It follows that $d(\lambda_1(0x),
0_S) = d(0x, 0_S) = d(\lambda_2(0x), 0_S)$ for all $\lambda_1,
\lambda_2 \in [0, \infty)$. By (iv), it can therefore not be true that
$0x \in S_0$. We find that $0x = 0_S$ for all $x \in S$. In addition, we necessarily have $\lambda 0_S = 0_S$ for all $\lambda \in [0, \infty)$; indeed, by the property just established, we have $\lambda 0_S = \lambda (0 \, 0_S) = (\lambda 0) 0_S = 0 \, 0_S = 0_S$.

We think of $S$ as `star-shaped' with rays emanating from the
origin. Alternatively, think of $S$ as a kind of cone. We will
sometimes write $x/\lambda := \lambda^{-1} x$ for $\lambda > 0$ and $x
\in S$. 

The distance function $x \mapsto d(x, 0_S)$ need not be
homogeneous. This will be important in
Section~\ref{sec:regul-vary-stoch}, where we will consider metrics on
sequence spaces inducing the product topology. To decompose a point in
$S_0$ in a `modular' component and an `angular' component, a modulus function needs to be present. The following definition captures the properties needed on such a function.

\begin{definition}
\label{def:radius}
A function $\rho : S \to [0, \infty)$ is a \emph{{modulus}} if it satisfies the following properties:
\begin{enumerate}[(i)]
\item 
  $\rho$ is continuous;
\item 
  $\rho$ is homogeneous: $\rho(\lambda x) = \lambda \, \rho(x)$ for $\lambda \in [0, \infty)$ and $x \in S$;
\item 
  for every $\eps > 0$, we have $\inf \{ \rho(x) : d(x, 0_S) > \eps \} > 0$. 
\end{enumerate}
\end{definition}

Since $\lambda 0_S = 0_S$ for all $\lambda \in [0, \infty)$, homogeneity implies $\rho(0_S) = 0$. The third
condition on the modulus $\rho$ will be needed to ensure that every
subset of $S_0$ which is bounded away from the origin is 
included in a set of the form $\{ x : \rho(x) \ge \delta \}$ for some
$\delta > 0$. In particular, $\rho(x) > 0$ for $x \neq 0_S$. Therefore,
$x = 0_S$ if and only if $\rho(x)=0$. Furthermore, the third condition implies that there
exist positive scalars $(z_r)_{r>0}$ such that $\lim_{r\downarrow 0}z_r=0$ and $\{x:\rho(x)<r\}\subset
\{ x:d(x,0_S)<z_r\}$ for every $r>0$. Since $\rho(0_S) = 0$ and since $\rho$ is continuous, the collection of sets $\{ x: \rho(x) < r \}$, for $r > 0$, therefore forms an open neighbourhood base for $0_S\in S$. 

We think of $\rho(x)$ as the `modulus' of $x$. We further define the
`angle' of $x \in S_0$ as $\theta(x) = \rho(x)^{-1} x$. Note that
$\rho(\theta(x)) = 1$, that is, $\theta(x) \in \{ \theta \in S :
\rho(\theta) = 1 \} =: \aleph$, the `unit sphere' of $S$. Clearly, $x
= \rho(x) \, \theta(x)$ for $x \in S_0$. The map
\[
  T : S_0 \to (0, \infty) \times \aleph : x \mapsto T(x) = (\rho(x), \theta(x))
\]
is the polar decomposition.

\begin{example}\label{ex:111}
In case the function $x \mapsto d(x, 0_S)$ is itself homogeneous, it is a modulus as in Definition~\ref{def:radius}. This is the case, for instance, if $S$ is a Banach space and the distance is the one induced by the norm, which brings us back to the set-up in \cite{SER10}. Another example is the Skorohod space $D = D([0, 1], \mathbb{R}^d)$ of c\`adl\`ag functions $[0, 1] \to \mathbb{R}^d$ equipped with the $J_1$-metric: in that case, the zero element $0_D$ is the zero function, and the Skorohod distance of $x \in D$ to $0_D$ is given by $d(x, 0_D) = \sup_{t \in [0, 1]} \norm{ x(t) }$. Regular variation of $D$-valued random elements was considered in \cite{HL05}.
\end{example}

\begin{example}
Assume that, for all $\eps > 0$, there exists $\delta > 0$ such that $\{ x : d(\delta^{-1} x, 0_S) \le 1 \} \subset \{ x : d(x, 0_S) \le \eps \}$. Then it can be shown that the map $\rho : S \to [0, \infty)$ defined by
\[
  \rho(x) = 
  \begin{cases}
    \inf \{ \lambda \in (0, \infty) : d(\lambda^{-1} x, 0_S) \le 1 \} & \text{if $x \neq 0_S$,} \\
    0 & \text{if $x = 0_S$.}
  \end{cases}
\]
is a modulus as in Definition~\ref{def:radius}. Intuitively, the condition on the metric $d$ is that scalar multiplication increases distances to the origin in a uniform way.
\end{example}

\begin{example}
Let $\domain$ be a nonempty compact subset of some Euclidean space and let $S = \USC_+(\domain)$ be the space of upper semicontinuous functions $x : \domain \to [0, \infty)$. Each such function $x$ is identified with its \emph{hypograph}, i.e., the set $\hypo x = \{ (\alpha, s) \in \reals \times \domain : \alpha \le x(s) \}$, a closed subset of $\reals \times \domain$. The \emph{hypo-topology} on $\USC_+(\domain)$ is the one induced by the Fell hit-and-miss topology on the space of closed subsets of $\reals \times \domain$; see \citet[Section~5.3]{molchanov:2005} for the dual theory of epi-convergence of lower semicontinuous functions. The map $\rho(x) = \sup_{s \in \domain} x(s)$, for $x \in \USC_+(\domain)$, then defines a modulus on $\USC_+(\domain)$.
\end{example}

If $S$ is locally compact, then condition~(iii) in Definition~\ref{def:radius} may be relaxed to the seemingly weaker assumption that $\rho(x) > 0$ for all $x \neq 0_S$. In general, however, the latter condition does not imply (iii); see Example~\ref{ex:Hilbert}. See also Section~\ref{sec:disc} for a discussion on condition~(iii).

\begin{example}
\label{ex:Hilbert}
Let $\mathbb{H}$ be a separable, infinite-dimensional Hilbert space, the metric being the one induced by the scalar product. Let $e_1, e_2, \ldots$ be an orthonormal basis in $\mathbb{H}$, and define $\rho(x) = (\sum_{i \ge 1} \lambda_i \abs{ \langle x, e_i \rangle }^2)^{1/2}$, where $(\lambda_i)_{i \ge 1}$ is a positive scalar sequence such that $\lambda_i \to 0$ as $i \to \infty$. Then $\rho$ satisfies conditions~(i) and~(ii) in Definition~\ref{def:radius}, and $\rho(x) > 0$ as soon as $x \neq 0_{\mathbb{H}}$. Still, condition (iii) in Definition~\ref{def:radius} is not satisfied, since $\rho(e_i) \to 0$ as $i \to \infty$ while $d(e_i, 0_{\mathbb{H}}) = 1$ for every $i \ge 1$.
\end{example}

\section{Regular variation and the polar decomposition}
\label{sec:rv}

Let $(S,d)$ be a complete, separable metric space equipped with an origin $0_S \in S$ and a scalar multiplication (Definition~\ref{def:mult}). Let $\mathcal{B}(S)$ denote the Borel $\sigma$-field on $S$ and let $M_0(S)$ be the space of Borel measures on $S_0 = S \setminus \{ 0_S \}$ that are bounded on complements on neighbourhoods of the origin. Let $\mathcal{C}_0$ denote the collection of bounded and continuous functions $f:S_0 \to \mathbb{R}$ for which there exists $r>0$ such that $f$ vanishes on $B_{0,r}=\{x\in S: d(x,0_S)<r\}$. The convergence of measures $\mu_n\to \mu$ in $M_0(S)$ holds as said in \cite{HL06} if and only if $\int f\,d\mu_n\to \int f\, d\mu$ for all $f\in \mathcal{C}_0$. Versions of the Portmanteau and continuous mapping theorems for $M_0$-convergence are stated as Theorems~2.4 and~2.5, respectively, in \cite{HL06}. 

For $\tau \in \reals$, let $\RV_\tau$ denote the class of regularly varying functions at infinity with index $\tau$, i.e., positive, measurable functions $g$ defined in a neighbourhood of infinity such that $\lim_{u \to \infty} g(\lambda u) / g(u) = \lambda^\tau$ for every $\lambda \in (0, \infty)$.

\begin{definition}[\citet{HL06}]
A random element $X$ in $S$ is \emph{regularly varying} with index $\alpha \in (0, \infty)$ if and only if there exists a function $V \in \RV_{-\alpha}$ and a nonzero measure $\mu \in M_0(S)$ such that
\[
  \frac{1}{V(u)} \Pr[ u^{-1} X \in \point ]
  \to
  \mu( \point ), \qquad u \to \infty.
\]
\end{definition}

The measure $\mu$ must be homogeneous: $\mu(\lambda \point) = \lambda^{-\alpha} \, \mu(\point)$ for every $\lambda \in (0, \infty)$ \citep[Theorem~3.1]{HL06}. 

\bgroup
Let $\rho$ be a modulus on $S$ (Definition~\ref{def:radius}). Our aim is now to extend to the present set-up the familiar decomposition of a regularly varying random vector into a regularly varying `modulus' and an asymptotically independent `angle'. First, we need a preliminary result linking the auxiliary function $V$ to the tail function $u \mapsto \Pr[\rho(X) > u]$.

\begin{lemma}
\label{lem:RV:aux}
Let $X$ be a regularly varying random element in $S$ with index $\alpha$ and limit measure $\mu$. If $\rho$ is a modulus on $S$, then $\mu( \{ x \in S : \rho(x) = \lambda \} ) = 0$ for every $\lambda \in (0, \infty)$ and
\[
  \lim_{u \to \infty} \frac{1}{V(u)} \Pr[ \rho(X) > u ] = \mu( \{ x \in S : \rho(x) > 1 \} ) \in (0, \infty).
\]
\end{lemma}

\begin{proof}
We have $\{ x : \rho(x) = \lambda \} = \{ \lambda x : \rho(x) = 1 \}$ for $\lambda\in (0,\infty)$. The set $\{ x : \rho(x)= \lambda \}$ is closed due to the continuity of $\rho$ and does not contain the origin. Hence $\mu( \{ x : \rho(x) = \lambda \}) =\lambda^{-\alpha} \, \mu( \{ x : \rho(x) = 1 \})$ must be finite. Since the sets $\{ x : \rho(x) = \lambda \}$ are disjoint for different $\lambda$, we conclude that $\mu( \{ x : \rho(x) = \lambda\} ) = 0$ for all $\lambda \in (0, \infty)$. 

The set $\{ x : \rho(x) = 1 \}$ is the boundary of the open set of $\{ x : \rho(x) > 1 \}$. The latter is thus a $\mu$-continuity set, and its closure, $\{ x : \rho(x) \ge 1 \}$, does not contain the origin, so that $\mu(\{ x : \rho(x) \ge 1 \}) < \infty$. We obtain
\begin{align*}
  \frac{1}{V(u)} \Pr[ \rho(X) > u ]
  &= \frac{1}{V(u)} \Pr[ \rho(u^{-1} X) > 1 ] \\
  &\to \mu ( \{ x : \rho(x) > 1 \} ) = \mu( \{ x : \rho(x) \ge 1 \} ), \qquad u \to \infty.
\end{align*}
The latter quantity must be nonzero: indeed, $\mu$ is nonzero and we have $S_0 = \bigcup_{k \ge 1} \{ x : \rho(x) > k^{-1} \}$ and $\mu(\{x : \rho(x) > k^{-1}\}) = \mu(k^{-1} \{ x : \rho(x) > 1 \}) = k^\alpha \, \mu(\{ x : \rho(x) > 1 \})$.
\end{proof}
\egroup

Let the arrow $\dto$ denote convergence in distribution, and let $\mathcal{L}(Y \mid A)$ denote the law of a random object $Y$ conditionally on an event $A$. For $\alpha > 0$, let $\Pareto(\alpha)$ denote the probability distribution of a random variable $Y$ such that $\Pr(Y > y) = y^{-\alpha}$ for $y \in [1, \infty)$. Recall $T(x) = (\rho(x), \theta(x))$ with $\theta(x) = \rho(x)^{-1} x$ for $x \in S_0$ and recall $\aleph = \{ x \in S : \rho(x) = 1 \}$. Let $\otimes$ signify product measure and let $\1_B$ denote the indicator function of a set $B$. 

\begin{proposition}\label{prop:spec00}
Let $X$ be a random element in $S$ and let $\alpha \in (0, \infty)$. Assume that a modulus $\rho:S\to [0,\infty)$ exists. The following properties are equivalent:
\begin{enumerate}[(i)]
\item
$X$ is regularly varying with index $\alpha > 0$.
\item
The function $u \mapsto \Pr[ \rho(X) > u ]$ is in $\RV_{-\alpha}$ and
there exists a probability measure $H$ on $\aleph=\{x\in S:\rho(x)=1\}$ such that
\begin{equation}
\label{eq:to_H}
\mathcal{L}  ( \theta(X) \mid \rho(X) > u) \dto H, 
  \qquad u \to \infty.
\end{equation}
\item
There exists a probability measure $H$ on $\aleph$ such that
\begin{equation}
\label{eq:to_P_H}
 \mathcal{L} ( \rho(X)/u, \theta(X) \mid \rho(X) > u )
  \dto
  \Pareto(\alpha) \otimes H,
  \qquad u \to \infty.
\end{equation}
\end{enumerate}
In that case, we have
\begin{equation}
\label{eq:to_mu}
  \frac{1}{\Pr[ \rho(X) > u ]}
  \Pr[ u^{-1} X \in \point ]
  \to
  \mu, \qquad u \to \infty,
\end{equation}
where $\mu$ is determined by
\begin{equation}
\label{eq:polar_decomposition}
  \mu \circ T^{-1}(dr, d\theta) = \alpha r^{-\alpha-1} dr \, H(d\theta), 
  \qquad (r, \theta) \in (0, \infty) \times \aleph.
\end{equation}
\end{proposition}

In terms of integrals, equation~\eqref{eq:polar_decomposition} means that, for $\mu$-integrable functions $f : S_0 \to \reals$, we have
\begin{equation*}
  \int_{S_0} f(x) \, d\mu(x)
  =
  \int_{r=0}^\infty \int_{\theta \in \aleph} f(r\theta) \, dH(\theta) \, \alpha r^{-\alpha-1} \, dr.
\end{equation*}

Proposition~\ref{prop:spec00} is to be compared with Corollary~4.4 in \cite{LRR14}. In the latter paper, the set excluded from the space $S$ is not necessarily just a single point but is allowed to be a closed cone. In contrast, the metric in \cite{LRR14} is supposed to be homogeneous, as in our Example~\ref{ex:111}, an assumption that we avoid here.


\begin{proof}[Proof of Proposition~\ref{prop:spec00}]
We break down the equivalence claim into a number of implications.

\emph{(i) implies (ii) and (iii).}
Let $\tilde{V}$ and $\tilde{\mu}$ be the auxiliary function and the limit measure in the definition of regular variation of $X$. By Lemma~\ref{lem:RV:aux}, we have $\Pr[ \rho(X) > u ] / \tilde{V}(u) \to \tilde{\mu}( \{ x : \rho(x) > 1 \} )$ as $u \to \infty$, the limit being finite and nonzero. Hence, the function $V(u) := \Pr[ \rho(X) > u ]$ is a valid auxiliary function for $X$ too. With this choice, the limit measure is then just a rescaled version of the old one: $\mu( \point ) = \tilde{\mu}(\point) / \tilde{\mu}( \{ x : \rho(x) > 1 \} )$. In particular, $\mu( \{ x : \rho(x) > 1 \} ) = 1$.

Define a Borel measure $H$ on $\aleph$ by
\[
  H(\point) = \mu ( \{ x : \rho(x) > 1, \, \theta(x) \in \point \} ).
\]
By construction, $H(\aleph) = 1$, i.e., $H$ is a probability measure.

For $r \in (0, \infty)$ and Borel sets $B \subset \aleph$, we have
\begin{align*}
  \mu \circ T^{-1}((r, \infty) \times B)
  &=
  \mu( \{ x : \rho(x) > r, \, \theta(x) \in B \} ) \\
  &=
  \mu( r \, \{ x : \rho(x) > 1, \, \theta(x) \in B \} ) \\
  &=
  r^{-\alpha} \, \mu( \{ x : \rho(x) > 1, \, \theta(x) \in B \} ) \\
  &=
  r^{-\alpha} \, H(B).
\end{align*}
Since the collection of sets of the form $\{ (r, \infty) \times B : r \in (0, \infty), B \in \Borel(\aleph) \}$ is a $\pi$-system generating the Borel $\sigma$-field on $(0, \infty) \times \aleph$, we find \eqref{eq:polar_decomposition}.

We prove \eqref{eq:to_H}. For $G \subset \aleph$ open, we have
\begin{align*}
  \liminf_{u \to \infty} \Pr[ \theta(X) \in G \mid \rho(X) > u ]
  &=
  \liminf_{u \to \infty} \frac{\Pr[ u^{-1} X \in T^{-1}( (1, \infty) \times G ) ]}{\Pr[ \rho(X) > u ]} \\
  &\ge
  \mu \circ T^{-1}( (1, \infty) \times G )
  = H(G).
\end{align*}
The inequality on the second line follows from the Portmanteau theorem for $M_0$ convergence \citep[Theorem~2.4]{HL06}: indeed, the set $T^{-1}((1, \infty) \times G)$ is open in $S$ and its closure does not contain the origin. The fact that the above display implies \eqref{eq:to_H} follows from the Portmanteau theorem for weak convergence of probability measures on metric spaces \citep[Theorem~2.1]{Billingsley:1999uc}.

Further, for $\lambda \in [1, \infty)$, we have
\begin{equation}
\label{eq:to_P}
  \Pr[ \rho(X) / u > \lambda \mid \rho(X) > u ] 
  = \frac{V(\lambda u)}{V(u)} 
  \to \lambda^{-\alpha}, 
  \qquad u \to \infty.
\end{equation}
It follows that $\law( \rho(X) / u \mid \rho(X) > u ) \dto \Pareto(\alpha)$ as $u \to \infty$.

By \eqref{eq:to_H} and \eqref{eq:to_P}, it follows that the distributions $\law( \rho(X)/u, \theta(X) \mid \rho(X) > u )$ are asymptotically tight as $u \to \infty$. It remains to show that there is only a single accumulation point.

Let $B \in \Borel(\aleph)$ be a $H$-continuity set and let $I$ be the open interval $(\lambda_1, \lambda_2)$ with $1 \le \lambda_1 < \lambda_2 < \infty$. The set $T^{-1}(I \times B) \subset S_0$ is bounded away from the origin and is a continuity set with respect to $\mu$. It follows that
\begin{multline*}
  \Pr[ \rho(X) / u \in I, \, \theta(X) \in B \mid \rho(X) > u ]
  = \frac{1}{V(u)} \Pr[ u^{-1}X \in T^{-1}(I \times B) ] \\
  \to \mu \circ T^{-1}(I \times B)
  = (\lambda_1^{-\alpha} - \lambda_2^{-\alpha}) \, H(B)
  = (\Pareto(\alpha) \otimes H)(I \times B)
\end{multline*}
as $u \to \infty$. This fixes the value of $L(I \times B)$ for any law $L$ that can arise as the limit in distribution of a sequence $[ \rho(X)/u_n, \theta(X) \mid \rho(X) > u_n]$ where $u_n \to \infty$ as $n \to \infty$. The collection of such sets $I \times B$ forms a $\pi$-system generating $\Borel((1, \infty) \times \aleph)$. (Use the Lindel\"of property to write every open subset of the separable metric space $(1, \infty) \times \aleph$ as a countable union of sets of the form $I \times B$, with $B$ an open ball in $\aleph$ whose boundary is an $H$-null set.) It follows that all sequences $[\rho(X)/u_n, \theta(X) \mid \rho(X) > u_n]$ converge in distribution to the same limit.


\textit{(iii) implies (ii).}
Convergence in distribution \eqref{eq:to_H} is a consequence of convergence in distribution \eqref{eq:to_P_H} and the continuous mapping theorem. Moreover, for $\lambda \ge 1$, we have $\Pr[ \rho(X) > \lambda u ] / \Pr[ \rho(X) > u ] = \Pr[ \rho(X)/u > \lambda \mid \rho(X) > u ] \to \lambda^{-\alpha}$ as $u \to \infty$. It follows that the function $u \mapsto \Pr[ \rho(X) > u ]$ belongs to $\RV_{-\alpha}$.

\emph{(ii) implies (i).}
Define a measure $\mu$ on $S_0$ by
\[
  \mu(B) 
  = 
  \int_{r = 0}^\infty 
    \int_{\theta \in \aleph} 
      \1_B (r\theta) \, 
    dH(\theta) \, 
  \alpha r^{-\alpha-1} \, dr,
  \qquad B \in \Borel(S_0),
\]
i.e., $\mu$ is the push-forward of the product measure $\alpha r^{-\alpha-1} dr \, dH(\theta)$ on $(0, \infty) \times \aleph$ induced by the map $(0, \infty) \times \aleph \to S_0 : (r, \theta) \mapsto r\theta$.

The measure $\mu$ is finite on complements of neighbourhoods of the origin. Indeed, let $\eps > 0$. By assumption, there exists $\delta > 0$ such that $d(x, 0) > \eps$ implies that $\rho(x) > \delta$. Therefore, $\{ x : d(x, 0) > \eps \} \subset \{ x : \rho(x) > \delta \}$. The $\mu$-measure of the latter set is equal to $\delta^{-\alpha}$, and thus finite.

We show that \eqref{eq:to_mu} holds. Let $B \in \Borel(\aleph)$ be a $H$-continuity set, i.e., $H(\partial B) = 0$, where $\partial B$ denotes the topological boundary of $B$ in $\aleph$. Let $0 < \lambda < \infty$. Put $V(u) = \Pr[ \rho(X) > u ]$. By the Portmanteau theorem for weak convergence of probability measures,
\begin{align*}
  \lefteqn{
  \frac{1}{V(u)} \Pr[ u^{-1} X \in \{ x : \rho(x) > \lambda, \, \theta(x) \in B \} ]
  } \\
  &=
  \frac{V(\lambda u)}{V(u)}
  \Pr[ \rho(X)^{-1} X \in B \mid \rho(X) > \lambda u ] \\
  &\to
  \lambda^{-\alpha} \, H(B) 
  = \mu ( \{ x : \rho(x) > \lambda, \, \theta(x) \in B \} ), \qquad u \to \infty.
\end{align*}
Since the limit is continuous in $\lambda$ and since $\{x : \rho(x) > \lambda \} \subset \{x : \rho(x) \ge \lambda \} \subset \{ x : \rho(x) > (1-\eps) \lambda \}$ for every $\eps \in (0, 1)$, we find that the above display remains valid if we replace `$\rho(x) > \lambda$' by `$\rho(x) \ge \lambda$'. It follows that, for every open interval $I = (\lambda_1, \lambda_2)$ with $0 < \lambda_1 < \lambda_2 < \infty$ and for each $H$-continuity set $B \in \Borel(\aleph)$, we have
\begin{multline}
\label{eq:IB}
  \frac{1}{V(u)} 
  \Pr[ u^{-1} X \in \{ x : \rho(x) \in I, \, \theta(x) \in B \} ] \\
  \to
  \mu ( \{ x : \rho(x) \in I, \, \theta(x) \in B \} ), \qquad u \to \infty.
\end{multline}

Let $G \subset S$ be open and such that $0_S \notin G^-$. The set $\{ (r, \theta) \in (0, \infty) \times \aleph : r\theta \in G \}$ is open by continuity of the scalar multiplication map. For every $x \in G$, there exists $0 < \eps < \rho(x)$ such that the set $(\rho(x) - \eps, \rho(x) + \eps) \times \{ \theta \in \aleph : d(\theta, \theta(x)) < \eps \}$ is a subset of $\{ (r, \theta) : r\theta \in G \}$. By decreasing $\eps$ if needed, we can ensure that the ball $\{ \theta \in \aleph : d(\theta, \theta(x)) < \eps \}$ is a $H$-discontinuity set. Let $\eps(x)$ denote the value of $\eps$ thus obtained, depending on $x \in G$.

The sets $A(x) = \{ y \in S : \abs{ \rho(y) - \rho(x) } < \eps(x), \, d(\theta(y), \theta(x)) < \eps(x) \}$, for $x \in G$, are open subsets of $G$ and they cover $G$ as $x$ ranges over $G$. By the Lindel\"of property, there exists a countable subcover of $G$ by sets $A(x_i)$, say. Finite intersections of the sets $A(x_i)$ are of the form $\{ y : \rho(y) \in I, \theta(x) \in B \}$, where $I$ is an open interval of $(0, \infty)$, bounded away from $0$ and $\infty$, and $B$ is an $H$-continuity subset of $\aleph$. 

Fix $\eta > 0$. Since $\mu(G) < \infty$, we can find a finite number $k$ such that $\mu(G) \le \mu(\bigcup_{i=1}^k A(x_i)) + \eta$. Write $\mu_u(\point) = V(u)^{-1} \Pr[ u^{-1} X \in \point ]$. By \eqref{eq:IB} and the inclusion-exclusion formula, $\mu_u( \bigcup_{i=1}^k A(x_i) ) \to \mu( \bigcup_{i=1}^k A(x_i) )$ as $u \to \infty$. But then we have $\liminf_{u \to \infty} \mu_u( G ) \ge \mu(G) - \eta$. Since $\eta > 0$ was arbitrary, we can conclude that $\liminf_{u \to \infty} \mu_u(G) \ge \mu(G)$.

Finally, let $F \subset S$ be closed and such that $0_S \notin F$. Since the complement of $F$ is open, there exists $\eps > 0$ such that $d(x, 0_S) \le \eps$ implies $x \notin F$. Further, there exists $\delta > 0$ such that $\rho(x) \le \delta$ implies $d(x, 0) \le \eps$. As a consequence, $F \subset \{ x : \rho(x) > \delta \}$. Define $G = \{ x : \rho(x) > \delta \} \setminus F$. Then $G$ is open and $0_S \notin G^-$. From the previous paragraph, recall $\liminf_{u \to \infty} \mu_u(G) \ge \mu(G)$. Further, $\mu_u(G) = V(\delta u)/V(u) - \mu_u(F)$ and $\mu(G) = \delta^{-\alpha} - \mu(F)$. It follows that $\limsup_{u \to \infty} \mu_u(F) \le \mu(F)$. Conclude by the Portmanteau Theorem~2.4 in \cite{HL06}.
\end{proof}

\section{Regularly varying time series}
\label{sec:regul-vary-stoch}

Recall that $(S, d)$ is a complete, separable metric space equipped with an origin and a scalar multiplication. In this section, no modulus will be needed. For simplicity, let from now on the origin of $S$ be denoted simply by $0$ rather than by $0_S$.

Let $S^{\mathbb{Z}}$ be the space of all sequences
$(x_t)_{t\in\mathbb{Z}}$ with elements in $S$. For nonnegative integer
$m$, identify the set $S^{\{ -m,\ldots ,m\} }$ with the set $S^{2m+1}$, so
that we can write $(x_{-m},\ldots,x_m) \in S^{2m+1}$. The sets
$S^{\mathbb{Z}}$ and $S^{2m+1}$ are endowed with the respective product topologies,
and these topologies can be metrized by the metrics $d_{\infty}$ and $d_m$, respectively, where
\begin{align*}
d_{\infty}(\bm{x},\bm{y})&= \sum_{t\in \mathbb{Z}} 2^{-|t|}
  \frac{d(x_t,y_t)}{1+d(x_t,y_t)}\,, \quad \bm{x},\bm{y}\in
                                   S^{\infty}\,,\\
d_m(\bm{x},\bm{y})&= \sum_{t=-m}^m 2^{-|t|}
                            \frac{d(x_t,y_t)}{1+d(x_t,y_t)}\,, \quad
                            \bm{x},\bm{y}\in S^{2m+1}\,.
\end{align*}
The metric spaces $(S^{\mathbb{Z}},d_{\infty})$ and $(S^{2m+1},d_m)$
 are complete and separable too. The precise choice of the metrics is
 not essential, and we could have chosen equivalent ones, replacing
 $d(x_t,y_t)/(1+d(x_t,y_t))$ by $\min\{d(x_t,y_t),1\}$, for instance.

Let $\bm{0}\in S^{\mathbb{Z}}$ be the zero sequence and let $\bm{0}^{(m)}=(0,\ldots,0)\in S^{2m+1}$. 
These are the origins of the spaces $S^{\mathbb{Z}}$ and $S^{2m+1}$, respectively. Scalar multiplication on these spaces is defined componentwise.

Let $\bm{X}=(X_t)_{t\in \mathbb{Z}}$ be a discrete-time stochastic
process, not necessarily stationary, taking values in $S$. There are two ways of defining regular
variation of $\bm{X}$: either by viewing $\bm{X}$ as a random
element of $S^{\mathbb{Z}}$ or via its finite-dimensional distributions. 
According to the following theorem, these two definitions are
essentially equivalent. For integer $m\ge 0$, write
$\bm{X}^{(m)}=(X_{-m},\ldots,X_m)$, a random element in
$S^{2m+1}$.

\begin{theorem}
  \label{thm:main}
Let $(S,d)$ be a complete, separable metric space equipped with an origin
$0\in S$ and a scalar multiplication. Let $\bm{X}=(X_t)_{t\in
  \mathbb{Z}}$ be a stochastic process in $S$. Let $\alpha \in (0, \infty)$ and
$V\in \RV_{-\alpha}$. The following two statements are equivalent:
\begin{itemize}
\item[(a)] There exists $\mu^{(\infty)}\in M_{\bm{0}}(S^{\infty})$
  such that $\mu^{(\infty)}(\{\bm{x}:x_0\neq 0 \})>0$ and such
  that, as $u \to \infty$,
\begin{equation}\label{eq:seqrv1}
\frac{1}{V(u)}\Pr[u^{-1}\bm{X}\in \point] \to \mu^{(\infty)}(\point)
  \quad \text{ in } M_{\bm{0}}(S^{\infty}).
\end{equation}
\item[(b)] For each nonnegative integer $m$, there exists a non-zero
  $\mu^{(m)}\in M_{\bm{0}^{(m)}}$ such that, as $u \to \infty$, 
\begin{equation}\label{eq:seqrv2}
\frac{1}{V(u)}\Pr[u^{-1}\bm{X}^{(m)}\in \point]\to
  \mu^{(m)}(\point)\quad \text{ in } M_{\bm{0}^{(m)}}(S^{2m+1}).
\end{equation}
\end{itemize}
If $(X_t)_{t\in \mathbb{Z}}$ is strictly stationary, the condition $\mu^{(\infty)}(\{\bm{x}:x_0\neq 0 \})>0$ in (a) is equivalent to the condition that $\mu^{(\infty)}$ is non-zero. 
\end{theorem}

In case the two equivalent conditions of Theorem \ref{thm:main} hold,
we say that the stochastic process $(X_t)_{t\in \mathbb{Z}}$ is
\emph{regularly varying}.

\begin{proof}[Proof of Theorem~\ref{thm:main}]
Regarding the last statement: if $(X_t)_{t \in \ZZ}$ is strictly stationary, then the value of $\mu^{(\infty)}( \{ \bm{x} : x_t \ne 0 \} )$ does not depend on $t \in \ZZ$, and since $S^{\ZZ} \setminus \{ \mathbf{0} \} = \bigcup_{t \in \ZZ} \{ \bm{x} : x_t \ne 0 \}$, we find that $\mu^{(\infty)}( \{ \bm{x} : x_0 \ne 0 \} ) > 0$ if and only if $\mu^{(\infty)}$ is non-zero.

For integer $n \ge m \ge 0$, define the projections $Q_m : S^{\ZZ} \to S^{2m+1}$ and $Q_{n,m} : S^{2n+1} \to S^{2m+1}$ by
\begin{align*}
  Q_m( \bm{x}) &= (x_{-m}, \ldots, x_m), && \bm{x} \in S^\ZZ, \\
  Q_{n,m}( x_{-n}, \ldots, x_n ) &= (x_{-m}, \ldots, x_m), && (x_{-n}, \ldots, x_n) \in S^{2n+1}.
\end{align*}
Let $Q_m^{-1}$ and $Q_{n,m}^{-1}$ denote the usual inverse images, inducing maps from the Borel $\sigma$-field $\Borel(S^{2m+1})$ to the Borel $\sigma$-fields $\Borel(S^\ZZ)$ and $\Borel(S^{2n+1})$, respectively. The projections are continuous and we have $Q_m(\mathbf{0}) = Q_{n,m}(\mathbf{0}^{(n)}) = \mathbf{0}^{(m)}$.

Further, for scalar $u > 0$ and integer $m \ge 0$, define the measures
\begin{align*}
  \upsilon_u^{(m)}( \point )
  &=
  \frac{1}{V(u)} \Pr \bigl[ u^{-1} \bm{X}^{(m)} \in \point \bigr], \\
  \upsilon_u^{(\infty)}( \point )
  &=
  \frac{1}{V(u)} \Pr \bigl[ u^{-1} \bm{X} \in \point \bigr].
\end{align*}
For integer $n \ge m \ge 0$, we have $Q_{n,m}( \bm{X}^{(n)} ) = Q_m( \bm{X} ) = \bm{X}^{(m)}$. We obtain
\[
  \upsilon_u^{(n)} \circ Q_{n,m}^{-1} = \upsilon_u^{(\infty)} \circ Q_m^{-1} = \upsilon_u^{(m)}.
\]

\textit{(a) implies (b).}
By assumption, $\upsilon_u^{(\infty)} \to \mu^{(\infty)}$ as $u \to \infty$ in $M_{\mathbf{0}}( S^\ZZ )$. Theorem~2.5 in \cite{HL06} yields
\[
  \upsilon_u^{(m)} 
  = 
  \upsilon_u^{(\infty)} \circ Q_m^{-1} 
  \to 
  \mu^{(\infty)} \circ Q_m^{-1} 
  = 
  \mu^{(m)}, 
  \qquad \text{as $u \to \infty$}
\]
in $M_{\bm{0}^{(m)}}( S^{2m+1} )$. The measure $\mu^{(m)}$ is non-zero, since $\mu^{(m)}( \{ \bm{x} : x_0 \ne 0 \} ) = \mu^{(\infty)}( \{ \bm{x} : x_0 \ne 0 \} ) > 0$.

\textit{(b) implies (a).}
Since $\upsilon_u^{(n)} \circ Q_{n,m}^{-1} = \upsilon_u^{(m)}$ for integer $n \ge m \ge 0$, we have, letting $u \to \infty$ and using again \citet[Theorem~2.5]{HL06},
\begin{equation}
\label{eq:munmum}
  \mu^{(n)} \circ Q_{n,m}^{-1} = \mu^{(m)}.
\end{equation}
This self-consistency property of the measures $\mu^{(m)}$ suggests
the use of the Daniell--Kolmogorov extension theorem to construct a
Borel measure $\mu^{(\infty)}$ on $S^\ZZ \setminus \{ \mathbf{0} \}$
such that $\mu^{(\infty)} \circ Q_m^{-1} = \mu^{(m)}$. Care is needed,
however, since the measures $\mu^{(m)}$ are finite only on complements
of neighbourhoods of $\mathbf{0}^{(m)}$ in $S^{2m+1}$. Moreover, the
spaces $S^{2m+1} \setminus \{ \mathbf{0}^{(m)} \}$ are not product
spaces. A more delicate construction is therefore needed to obtain
$\mu^{(\infty)}$, starting from a decreasing sequence of neighborhoods
of the zero sequence $\bm{0}$ in
$S^\ZZ$. Convergence to the limit measure $\mu^{(\infty)}$ will then
be shown using Theorem~\ref{thm:hh1}. 


Let $\Borel_f(S^{\ZZ})$ be the class of cylinders of $S^\ZZ$, that is, 
\begin{equation}
\label{eq:Borel:f}
  \Borel_f(S^\ZZ) 
  = 
  \bigcup_{m=0}^\infty \{ Q_m^{-1}(A) : A \in \mathcal{B}(S^{2m+1}) \}.
\end{equation}
For integer $m \ge 0$ and for real $r > 0$, define
\begin{equation*}
  N_{m,r}(\bm{x}) = \{ \bm{y} \in S^{2m+1}: d_m(\bm{x}, \bm{y}) < r \}, \qquad
  \bm{x} \in S^{2m+1}.
\end{equation*}  
For $\bm{x}, \bm{y} \in S^{\ZZ}$ we have $d_\infty( \bm{x}, \bm{y} ) \le d_m( Q_m(\bm{x}), Q_m(\bm{y}) ) + 2^{-m}$. We obtain
\begin{eqnarray*}
  Q_m^{-1}(N_{m,r}(Q_m(\bm{x})))
  \subset
  \{ \bm{y} \in S^{\ZZ} : d_{\infty}(\bm{x}, \bm{y}) < r + 2^{-m} \}, 
  \qquad \bm{x} \in S^{\ZZ}.
\end{eqnarray*}
For every $\varepsilon > 0$ we can find $r > 0$ and integer $m \ge 1$
such that $r + 2^{-m} \le \varepsilon$. Therefore, we can write any
open subset of $S^\ZZ$ as a countable union of open cylinders: apply the Lindel\"{o}f property, using the separability of the metric space $(S^\ZZ, d_\infty)$. The $\sigma$-field generated by $\Borel_f(S^\ZZ)$ is thus equal to $\Borel(S^\ZZ)$. Clearly, $\Borel_f(S^\ZZ)$ is a $\pi$-system.

Fix integer $m \ge 0$ and let
\begin{equation}
\label{eq:Borel:s}
  \Borel_s( S^{2m+1} )
  =
  \{ B \in \Borel(S^{2m+1}) : \mu^{(m)}(\partial B) = 0 \},
\end{equation}
i.e., the collection of $\mu^{(m)}$-smooth Borel sets of $S^{2m+1}$. Since $\partial (A \cap B) \subset \partial A \cup \partial B$ for all subsets $A$ and $B$ of a topological space, $\Borel_s(S^{2m+1})$ is a $\pi$-system. Moreover, finiteness of $\mu^{(m)}$ on complements of neighbourhoods of $\mathbf{0}^{(m)}$ together with separability of the metric space $(S^{2m+1}, d_m)$ implies, via the Lindel\"{o}f property, that every open subset of $S^{2m+1}$ can be covered by a countable collection of $\mu^{(m)}$-smooth open balls. In particular, the $\sigma$-field generated by $\Borel_s(S^{2m+1})$ is equal to $\Borel(S^{2m+1})$.

For integer $m \ge k \ge 1$, define the subset $A_k^{(m)}$ of $S^{2m+1}$ by
\[
  A_k^{(m)}
  =
  \left\{
    \bm{x} \in S^{2m+1} : \max_{-k \le j \le k} d(x_j, 0) \le 1/k
  \right\}.
\]
By homogeneity of $\mu^{(m)}$, we have $\mu^{(m)}(\{ \bm{x} \in S^{2m+1} : d(x_j, 0) = c \}) = 0$ for all integer $m \ge 0$, $j\in \{-m,\ldots,m\}$, and real $c>0$. Therefore, $\mu^{(m)}( \partial A_k^{(m)} ) = 0$ for all integer $m \ge k \ge 1$. The set $S^{2m+1}\setminus A_k^{(m)}$ is $\mu^{(m)}$-smooth and open. By construction, $d_m(\bm{x},\mathbf{0}^{(m)}) \ge 2^{-k}/(1+k)$ for $\bm{x} \in S^{2m+1} \setminus A_k^{(m)}$ while $d_m(\bm{x}, \mathbf{0}^{(m)})\le 3/(1+k)$ for $\bm{x}\in A_k^{(m)}$. As a consequence of the former inequality, $\mu^{(m)}(S^{2m+1}\setminus A_k^{(m)}) < \infty$.

For integer $m \ge k \ge 1$, write
\begin{equation}
\label{eq:mumk}
  \mu^{(m)}_k = \mu^{(m)}( \point \setminus A_k^{(m)} ). 
\end{equation}
If additionally $n \ge m$, we have, since $Q_{n,m}^{-1}( A_k^{(m)} ) = A_k^{(n)}$, by \eqref{eq:munmum},
\begin{equation}
\label{eq:mumkmunk}
  \mu^{(n)}_k \circ Q_{n,m}^{-1} 
  = 
  \mu^{(m)}_k, \qquad n \ge m \ge k \ge 1.
\end{equation}

Let $R_k = \mu^{(k)}( S^{2k+1} \setminus A_k^{(k)} )$ be the common value of the total mass of the measures $\mu^{(m)}_k$ for $m \ge k$. Then $0 < R_k < \infty$: positivity follows from the fact that $\mu^{(k)}$ is nonzero and homogenous; finitess follows because $A_k^{(k)}$ is a neighbourhood of $\bm{0}^{(k)}$ in $S^{2k+1}$.

Fix integer $k \ge 1$. For integer $m \ge k$, consider the probability measures $P_k^{(m)} = R_k^{-1} \mu_k^{(m)}$ on $S^{2m+1}$. By \eqref{eq:mumkmunk}, we have
\begin{equation}
\label{eq:PmkPnk}
  P_k^{(n)} \circ Q_{n,m}^{-1}
  =
  P_k^{(m)},
  \qquad n \ge m \ge k.
\end{equation}
By \eqref{eq:PmkPnk}, the family $(P_k^{(m)})_{m \ge k}$ is consistent in the sense of \citet[Chapter~4, Section~8]{POL02}. Since the metric space $(S^{2m+1}, d_m)$ is separable and complete for all $m \ge k$, every probability measure $P_k^{(m)}$ is tight \citep[Theorem~1.3]{Billingsley:1999uc}. According to the Daniell--Kolmogorov extension theorem \citep[Theorem~53]{POL02}, there exists a tight probability measure $P_k^{(\infty)}$ on $S^\ZZ$ such that 
\begin{equation} 
\label{eq:Pkinf}
  P_k^{(\infty)} \circ Q_m^{-1} = P_k^{(m)}, \qquad m \ge k.
\end{equation}
Define $\mu_k^{(\infty)} = R_k \, P_k^{(\infty)}$. Then \eqref{eq:Pkinf} implies
\begin{equation}
\label{eq:muk}
  \mu_k^{(\infty)} \circ Q_m^{-1} = \mu_k^{(m)}, \qquad m \ge k.
\end{equation}

The sets
\[
  A_k^{(\infty)}
  = 
  \left\{ 
    \bm{x} \in S^{\ZZ} : \max_{-k \le j \le k} d(x_j, 0) \le k^{-1} 
  \right\},
  \qquad k \ge 1,
\]
form a decreasing sequence of closed neighbourhoods of $\mathbf{0}$ in $S^\ZZ$. Clearly, $A_k^{(\infty)} \subset \{ \bm{x} : d_\infty(\bm{x}, \mathbf{0}) \le 3/(1+k) \}$. For $B \subset S^\ZZ$ such that $\mathbf{0} \notin B^-$, there exists $k_0 \ge 1$ such that $A_k^{(\infty)} \cap B = \varnothing$ for all $k \ge k_0$.

By \eqref{eq:mumk} and \eqref{eq:muk}, the measure $\mu_k^{(\infty)}$ is finite and vanishes on $A_k^{(\infty)} = Q_k^{-1}(A_k^{(k)})$:
\begin{align}
\label{eq:muk:zero}
  \mu_k^{(\infty)} \bigl( A_k^{(\infty)} \bigr) 
  &= 
  \mu^{(k)} \bigl( A_k^{(k)} \setminus A_k^{(k)} \bigr) 
  =
  0, \\
\label{eq:muk:finite}
  \mu_k^{(\infty)} \bigl( S^\ZZ \setminus A_k^{(\infty)} \bigr)
  &=
  \mu^{(k)} \bigl( S^{2k+1} \setminus A_k^{(k)} \bigr)
  <
  \infty.
\end{align}
Moreover, we have
\begin{equation}
\label{eq:muellmuk}
  \mu_\ell^{(\infty)} \bigl( \point \setminus A_k^{(\infty)} \bigr)
  = \mu_k^{(\infty)} \bigl( \point \setminus A_k^{(\infty)} \bigr)
  = \mu_k^{(\infty)}(\point),
  \qquad \ell \ge k \ge 1.
\end{equation}
Indeed, for $m \ge \ell$ and $B \in \Borel(S^{2m+1})$, we have, by \eqref{eq:muk},
\begin{align*}
  \mu_\ell^{(\infty)} \bigl( Q_m^{-1}(B) \setminus A_k^{(\infty)} \bigr)
  &= \mu_\ell^{(\infty)} \bigl( Q_m^{-1}(B \setminus A_k^{(m)}) \bigr) \\
  &= \mu_\ell^{(m)}(B \setminus A_k^{(m)})
  = \mu_k^{(m)}(B)
  = \mu_k^{(\infty)}(Q_m^{-1}(B)),
\end{align*}
and the cylinders of $S^\ZZ$ form a $\pi$-system generating $\Borel(S^\ZZ)$; apply Theorem~3.3 in \citet{BIL95} to arrive at \eqref{eq:muellmuk}.

According to \eqref{eq:muellmuk}, the measures $\mu_k^{(\infty)}$ are successive extensions of one another, each measure being supported on $S^\ZZ \setminus A_k^{(\infty)}$, a sequence of subsets of $S^\ZZ$ which is growing to $S^\ZZ \setminus \{ \mathbf{0} \}$ as $k \to \infty$. These properties can be used to define a measure $\mu^{(\infty)}$ concentrated on $S^\ZZ \setminus \{ \mathbf{0} \}$ by
\[
  \mu^{(\infty)}(B)
  = \mu_1^{(\infty)} \bigl( B \setminus A_1^{(\infty)} \bigr)
  + \sum_{k=2}^\infty 
  \mu_k^{(\infty)} \bigl( (B \cap A_{k-1}^{(\infty)}) \setminus A_k^{(\infty)} \bigr),
  \qquad B \in \Borel(S^\ZZ).
\]
By properties \eqref{eq:muk:zero}, \eqref{eq:muk:finite} and \eqref{eq:muellmuk}, we have
\begin{equation}
\label{eq:muinf}
  \mu^{(\infty)}( \point \setminus A_k^{(\infty)} ) 
  = \mu_k^{(\infty)}( \point ),
  \qquad k \ge 1.
\end{equation}

The measures $\mu^{(m)}$ and $\mu^{(\infty)}$ are connected through the formula
\begin{equation}
\label{eq:mummuinf}
  \mu^{(\infty)} \circ Q_m^{-1} = \mu^{(m)}, \qquad m \ge 0.
\end{equation}
Indeed, let $B \in \Borel(S^{2m+1})$ be such that $\mathbf{0}^{(m)} \notin B^-$. Then we can find $\ell \ge \max(m, 1)$ such that $\max_{j=-m,\ldots,m} d(x_j, 0) > 1/\ell$ for all $\bm{x} \in B$ and thus $A_\ell^{(\infty)} \cap Q_m^{-1}(B) = \varnothing$ and $A_\ell^{(\ell)} \cap Q_{\ell,m}^{-1}(B) = \varnothing$. Since $Q_m^{-1}(B) = Q_\ell^{-1}(Q_{\ell,m}^{-1}(B))$, we find, applying successively equations \eqref{eq:muinf}, \eqref{eq:muk}, \eqref{eq:mumk} and \eqref{eq:munmum},
\begin{align*}
  \mu^{(\infty)} \bigl( Q_m^{-1}(B) \bigr)
  &=
  \mu_\ell^{(\infty)} \bigl( Q_m^{-1}(B) \bigr) \\
  &=
  \mu_\ell^{(\ell)} \bigl( Q_{\ell,m}^{-1}(B) \bigr) \\
  &=
  \mu^{(\ell)} \bigl( Q_{\ell,m}^{-1}(B) \bigr) \\
  &=
  \mu^{(m)}(B),
\end{align*}
as required. In particular, $\mu^{(\infty)}( \{ \bm{x} : x_0 \neq 0 \} ) = \mu^{(0)}( S \setminus \{ 0 \} ) > 0$.

To prove $M_0$-convergence in Theorem~\ref{thm:main}(a), we apply Theorem~\ref{thm:hh1}. Recall the $\mu^{(m)}$-smooth Borel sets $\Borel_s(S^{2m+1})$ in \eqref{eq:Borel:s}. Define the collection $\mathcal{A} \subset \Borel( S^\ZZ )$ by
\[
  \mathcal{A} 
  = 
  \bigcup_{m=0}^\infty
  \bigl\{ 
    Q_m^{-1}(B) : B \in \Borel_s(S^{2m+1}), \, \mathbf{0}^{(m)} \notin B^- 
  \bigr\}.
\]
We show that $\mathcal{A}$ satisfies conditions (C1) and (C2) of Theorem~\ref{thm:hh1}.
\begin{enumerate}[({C}1)]
\item
Put $N_i = A_i^{(\infty)}$ for $i \in \NN$. We have $d_\infty(\bm{x}, \mathbf{0}) < 3/(1+i) + 2^{-i}$ for $\bm{x} \in N_i$. Further, let $i \in \NN$ and let $A = Q_m^{-1}(B)$ with $B \in \Borel_s( S^{2m+1} )$ and $\mathbf{0}^{(m)} \notin B^-$, so $A \in \mathcal{A}$. We need to show that $A \setminus N_i \in \mathcal{A}$ too. Put $\ell = \max(m, i)$ and note that $A = Q_\ell^{-1}(Q_{\ell,m}^{-1}(B))$ and $N_i = Q_\ell^{-1}(Q_{\ell,i}^{-1}(A_i^{(i)}))$, whence 
\[
  A \setminus N_i 
  = 
  Q_\ell^{-1} \bigl( 
    Q_{\ell,m}^{-1}(B) \setminus Q_{\ell,i}^{-1}(A_i^{(i)})
  \bigr).
\]
The set on the right-hand side is of the desired form $Q_{\ell}^{-1}(C)$ for some set $C \in \Borel_s(S^{2\ell+1})$ such that $\mathbf{0}^{(\ell)} \notin C^-$; indeed, we have for instance $\mu^{(\ell)}(\partial Q_{\ell,m}^{-1}(B)) = \mu^{(\ell)}(Q_{\ell,m}^{-1}(\partial B)) = \mu^{(m)}(\partial B) = 0$ by \eqref{eq:mumkmunk} and the fact that $\mathbf{0}^{(m)} \notin B^-$. As a consequence, $A \setminus N_i \in \mathcal{A}$ for $A \in \mathcal{A}$.
\item
Recall from the paragraph containing \eqref{eq:Borel:f} that every
open subset of $S^\ZZ$ can be written as the union of a countable
collection of open cylinders. Moreover, recall from the paragraph containing \eqref{eq:Borel:s} that every open subset of $S^{2m+1}$ can be written as a countable collection of $\mu^{(m)}$-smooth open balls, and this for arbitrary integer $m \ge 0$. As a consequence, every open subset $G$ of $S^\ZZ$ such that $\mathbf{0} \notin G^-$ can be written as a countable union of $\mathcal{A}$-sets.
\end{enumerate}

Finally, by Theorem~\ref{thm:main}(b) and the Portmanteau theorem \citep[Theorem~2.4]{HL06}, we have, for every $A = Q_m^{-1}(B) \in \mathcal{A}$ with $B \in \Borel_s(S^{2m+1})$ and $\mathbf{0}^{(m)} \notin B^-$,
\begin{align*}
  \frac{1}{V(u)} \Pr[u^{-1} \bm{X} \in A ]
  &=
  \frac{1}{V(u)} \Pr[u^{-1} \bm{X}^{(m)} \in B] \\
  &\to
  \mu^{(m)}(B)
  =
  \mu^{(\infty)}(A),
  \qquad u \to \infty,
\end{align*}
the final identity following from \eqref{eq:mummuinf}. Apply Theorem~\ref{thm:hh1} to conclude that the $M_0$-convergence stated in Theorem~\ref{thm:main}(a) holds.
\end{proof}

\section{Angular or spectral tail processes}
\label{sec:spectral-processes}

Let $\bm{X}=(X_t)_{t\in \mathbb{Z}}$ be a strictly stationary, regularly varying discrete-time stochastic process taking values in $S$. According to Theorem~\ref{thm:main}, the random variable $X_0$ is a regularly varying random element in the space $(S,d)$. With the assumption that a modulus $\rho:S\to [0,\infty)$ exists, Proposition~\ref{prop:spec00} describes the joint limit of the rescaled modulus $\rho(X_0)/u$ and the angle $\theta(X_0) = X_0 / \rho(X_0)$ given that $\rho(X_0) > u$ as $u \to \infty$. In the following theorem, we will extend this by considering the entire self-normalized process $X_t / \rho(X_0)$, $t \in \mathbb{Z}$. The theorem generalizes Theorem~2.1 in \citet{BS09} and Theorem~3.1 in \citet{SER10}. Let $\mathcal{C}_0(S^k)$ be the space of functions $S^k \setminus \{(0, \ldots, 0)\} \to \reals$ which are bounded and continuous and vanish on the complement of a neighbourhood of the origin, $(0, \ldots, 0)$, in $S^k$. Recall that the arrow $\dto$ signifies weak convergence.

\begin{theorem}
\label{thm:spectralprocess}
Let $\bm{X}=(X_t)_{t\in \mathbb{Z}}$ be a strictly stationary time series taking values in a complete, separable metric space $S$, endowed with an origin, a scalar multiplication, and a modulus $\rho$. The following properties are equivalent: 
\begin{itemize}
\item[(i)]
  $\bm{X}$ is regularly varying with index $\alpha \in (0, \infty)$.  

\item[(ii)] 
  The function $u\mapsto \Pr[\rho(X_0)>u]$ belongs to $\RV_{-\alpha}$ and there exists a random element $(\Theta_t)_{t\in \mathbb{Z}}\in S^{\mathbb{Z}}$ such that, as $u \to \infty$,
\begin{equation}
\label{eq:spp01}
\mathcal{L}((X_t/\rho(X_0))_{t\in \mathbb{Z}}\mid \rho(X_0)>u)\dto
  (\Theta_t)_{t\in \mathbb{Z}}.
\end{equation} 

\item[(iii)] 
  There exist a $\Pareto(\alpha)$ random variable $Y$ and a random element $(\Theta_t)_{t\in \mathbb{Z}}\in S^{\mathbb{Z}}$, independent of each other, such that, as $u \to \infty$,
\begin{equation}
\label{eq:spp02}
\mathcal{L}(\rho(X_0)/u, (X_t/\rho(X_0))_{t\in \mathbb{Z}}\mid
\rho(X_0)>u)\dto (Y,(\Theta_t)_{t\in \mathbb{Z}}).
\end{equation}
\end{itemize}
In this case, the law of $(\Theta_t)_{t\in \mathbb{Z}}$ is the same across (ii) and (iii), and for every integer $t$ and every positive integer $k$,
\begin{eqnarray}
\label{eq:claim1}
&& \E [ \rho(\Theta_t)^\alpha ] 
  = \lim_{r \downarrow 0} \lim_{u \to \infty} \Pr[ \rho(X_{-t}) > ru
   \mid \rho(X_0) > u ]\le 1\,,\\
\label{eq:spp03}
&&\frac{1}{\Pr[\rho(X_0)>u]} \Pr[(X_1/u,\ldots,X_k/u)\in \point]\to
\nu_k(\point)\,, \qquad u\to \infty\,,
\end{eqnarray}
in $M_0(S^k)$, where $\int f\,d\nu_k$ for $f\in \mathcal{C}_0(S^k)$ is given by
\begin{equation}
\label{eq:intfnuk}
  \sum_{i=1}^k
  \int_0^{\infty}
    \E \biggl[ 
      f(0,\ldots,0,z\Theta_0,\ldots,z\Theta_{k-i}) \, 
      \1 \biggl( \max_{1-i\le j\le -1}\rho(\Theta_j)=0 \biggr) 
    \biggr]\,
  d(-z^{-\alpha})\,.
\end{equation}
\end{theorem}

\begin{proof}
We prove the implications (i) $\implies$ (ii) $\implies$ (iii) $\implies$ (i) \& \eqref{eq:claim1} \& \eqref{eq:spp03}.

\emph{(i) implies (ii).}
Let $\widetilde{V}$, $\widetilde{\mu}^{(\infty)}$ and $\widetilde{\mu}^{(m)}$ be the auxiliary function and the limit measures, respectively, in \eqref{eq:seqrv1} and \eqref{eq:seqrv2}. Proposition~\ref{prop:spec00} and the equation \eqref{eq:seqrv2} imply that when $m=0$, we have
\[ 
  \frac{1}{\widetilde{V}(u)}\Pr[\rho(X_0)>u]\to \widetilde{\mu}^{(0)}(\{x\in S:\rho(x)>1\}), 
  \qquad u\to \infty,
\]
the limit being finite and nonzero. Hence, the function $V(u):=\Pr[\rho(X_0)>u]$ is a valid auxiliary function for $\bm{X}$. With this choice, the limit measures are rescaled versions of the old ones:
\begin{align*}
  \mu^{(\infty)}(\point)
  &= \widetilde{\mu}^{(\infty)}(\point)/\widetilde{\mu}^{(0)}(\{x:\rho(x)>1\}),\\ 
  \mu^{(m)}(\point)
  &= \widetilde{\mu}^{(m)}(\point)/\widetilde{\mu}^{(0)}(\{x:\rho(x)>1\}).
\end{align*}

For every $\lambda>0$, the homogeneity of the measure $\mu^{(\infty)}$ and the modulus $\rho$ implies that $\mu^{(\infty)}(\{\bm{x}\in S^{\mathbb{Z}}:\rho(x_0)=\lambda \})=\lambda^{-\alpha}\mu^{(\infty)}(\{\bm{x}\in S^{\mathbb{Z}}:\rho(x_0)=1\})$. The set $\{\bm{x}\in S^{\mathbb{Z}}:\rho(x_0)=1\}$ is the boundary of the open set $\{\bm{x}\in S^{\mathbb{Z}}:\rho(x_0)>1\}$. The latter is thus a $\mu^{(\infty)}$-continuity set, and its closure, $\{\bm{x}\in S^{\mathbb{Z}}:\rho(x_0)\ge 1\}$, does not contain the origin $\bm{0}$. Similarly, we have that, for every nonnegative integer $m$, the set $\{\bm{x}^{(m)}=(x_{-m},\ldots,x_m)\in S^{2m+1}:\rho(x_0)>\lambda)\}$ is a $\mu^{(m)}$-continuity set, whose closure does not contain the origin $\bm{0}^{(m)}$.

Let $m$ be a nonnegative integer and write $k=2m+1$. Put 
\begin{equation}\label{eq:nn}
  \aleph_m=\{(\theta_{-m},\ldots,\theta_m)\in S^k:\rho(\theta_0)=1 \}
\end{equation}
and define a probability measure on $\aleph_{m}$ by
\begin{equation*}
  H_{m}(B)=\mu^{(m)}(\{\bm{x}^{(m)}=(x_{-m},\ldots,x_m)\in S^k:\rho(x_0)>1,
  \bm{x}^{(m)}/\rho(x_0)\in B \})\,, 
\end{equation*}
for Borel sets $B\subset\aleph_{m}$. Let $g:\aleph_{m}\to \mathbb{R}$ be bounded and continuous and define $f:S^k\to \mathbb{R}$ by
 \begin{equation*}
   f(x_{-m},\ldots,x_m)
   =
   g(x_{-m}/\rho(x_0),\ldots,x_m/\rho(x_0)) \, \1\{\rho(x_0)>1\},
 \end{equation*}
to be interpreted as $0$ if $x_0=0$. The function $f$ is bounded and vanishes on the set $\{ \bm{x}^{(m)} \in S^{2m+1} : \rho(x_0) \le 1 \}$, which is a closed neighbourhood of the origin $\bm{0}^{(m)}$ in $S^{2m+1}$. Moreover, it is
continous everywhere except perhaps on $\aleph_m$, which is a
$\mu^{(m)}$-null set. 
By Lemma~\ref{lem:hh1}, 
\begin{align*}
&\E[g(X_{-m}/\rho(X_0),\ldots,X_m/\rho(X_0))\mid \rho(X_0)>u]\\ 
 &\quad =
   \frac{1}{\Pr[\rho(X_0)>u]}\E[f(X_{-m}/\rho(X_0),\ldots,X_m/\rho(X_0))]\\ 
&\quad \to \int_{S^k}f\,d\mu^{(m)}=\int_{\aleph_m}g\,dH_m\,,
  \quad u\to \infty\,.
\end{align*}
If $(\Theta_{-m},\ldots,\Theta_m)$ is a random element of $\aleph_m$
with distribution $H_m$, then
\begin{equation*}
  \law\Bigl( \big(X_{-m}/\rho(X_0),\ldots,X_m/\rho(X_0)\big) \mid \rho(X_0)>u \Bigr)
  \dto (\Theta_{-m},\ldots,\Theta_m),
\end{equation*}
as $u \to \infty$. The Daniell--Kolmogorov extension theorem \citep[Theorem 53]{POL02}
yields that there exists a random element $(\Theta_t)_{t\in
  \mathbb{Z}}$ in $S^{\mathbb{Z}}$ such that, for every nonnegative integer $m$, the distribution of
$(\Theta_{-m},\ldots,\Theta_m)$ is $H_m$. Weak convergence of finite stretches characterizes weak
convergence in the product space $S^{\mathbb{Z}}$ \citep[Theorem~1.4.8]{VDV96}, and statement (ii) follows. 

\emph{(ii) implies (iii).} 
Fix a nonnegative integer $m$. Let $y\ge 1$ and let $g:\aleph_m\to \mathbb{R}$ be continous and bounded, with $\aleph_m$ as in \eqref{eq:nn}. We have 
\begin{align*}
&\E[\1\{\rho(X_0)/u>y\} \, g(X_{-m}/\rho(X_0),\ldots,X_m/\rho(X_0))\mid
  \rho(X_0)>u]\\
&\quad =\frac{\Pr[\rho(X_0)>uy]}{\Pr[\rho(X_0)>u]}
  \E[g(X_{-m}/\rho(X_0),\ldots,X_m/\rho(X_0))\mid \rho(X_0)>uy]\\ 
&\quad \to y^{-\alpha}\E[g(\Theta_{-m},\ldots,\Theta_m)]\,, \quad u\to \infty\,.
\end{align*}
In view of Lemma~\ref{L:product}, as $u\to \infty$,
\[
  \law\Bigl( (\rho(X_0)/u,X_{-m}/\rho(X_0),\ldots,X_m/\rho(X_0))\mid \rho(X_0)>u \Bigr)
  \dto (Y,\Theta_{-m},\ldots,\Theta_m),
\]
where $Y$ is a $\Pareto (\alpha)$ random variable independent of $(\Theta_{-m},\ldots,\Theta_m)$. Statement~(iii) follows. 
 
\emph{(iii) implies (i), \eqref{eq:claim1} and \eqref{eq:spp03}.} 
To prove (i), we will show that \eqref{eq:seqrv2} holds with $V(u) = \Pr[ \rho(X_0) > u ]$. The latter function belongs to $\RV_{-\alpha}$ because (iii) implies that $V(uy) / V(u) = \Pr[ \rho(X_0)/u > y \mid \rho(X_0) > u ] \to y^{-\alpha}$ as $u \to \infty$, for all $y \ge 1$. By \citet[Theorem~2.1]{HL06}, equation~\eqref{eq:seqrv2} is equivalent to the condition that for every $m\ge 0$ and every $f\in \mathcal{C}_{0}(S^{2m+1})$, we have
\begin{equation*}
\lim_{u\to \infty} \frac{1}{V(u)}\E[f(X_{-m}/u,\ldots,X_m/u)]=\int_{S^{2m+1}}f(x_{-m},\ldots,x_m)\,d\mu^{(m)}\,.
\end{equation*}
By stationarity, this limit relation is a consequence of \eqref{eq:spp03}: just replace $(X_{-m}, \ldots, X_m)$ by $(X_1, \ldots, X_k)$ with $k = 2m+1$.


We start with proving \eqref{eq:claim1}. Fix integer $t$ and real $r>0$. Put $V(u) = \Pr[\rho(X_0) > u]$. Statement (iii) implies the independence between $Y$ and $(\Theta_t)_{t\ge 0}$. Writing
\begin{equation*}
  \frac{X_t}{u}=r \frac{\rho(X_0)}{ur} \,\frac{X_t}{\rho(X_0)}\,,
\end{equation*}
we have, by stationarity and Fubini's theorem,
\begin{align*}
  \lefteqn{
    \lim_{u \to \infty} \Pr [\rho(X_{-t}) > ru \mid \rho(X_0)>u]
  } \\
  &=
    \lim_{u\to \infty}\frac{V(ru)}{V(u)} 
    \Pr [\rho(X_t)>u \mid \rho(X_0)>ru] \\
  &=
    r^{-\alpha} \Pr[rY\rho(\Theta_t)>1]
    = \E \bigg[
      r^{-\alpha}
      \int_1^{\infty} \1 \{ ry\rho(\Theta_t)>1 \} \,d(-y^{-\alpha}) 
      \bigg] \\ 
  &=
    \E \bigg[\int_r^{\infty}\1 \{z\rho(\Theta_t)>1\}\,d(-z^{-\alpha}) \bigg]
    = \E [\min\{\rho(\Theta_t),r^{-1}\}^{\alpha}].
\end{align*}
By monotone convergence, we have $\E [\min\{\rho(\Theta_t),r^{-1}\}^{\alpha}] \to \E[\rho(\Theta_t)^{\alpha}]$ as $r \downarrow 0$, whence~\eqref{eq:claim1}. 

Fix $f \in \mathcal{C}_0(S^{k})$. There exists $r_0>0$ such that $f$ vanishes on the set $\{ \bm{x} \in S^k : \max_{1 \le i \le k} \rho(x_i) \le r_0 \}$. Indeed, $f$ vanishes on a neighbourhood of the origin $(0, \ldots, 0)$ in $S^k$ and sets of the stated form constitute a neighbourhood basis of this origin, by Definition~\ref{def:radius}(iii) and by definition of the product topology.

Fix $r \in (0, r_0)$. Form a partition of $\{\max_{1\le i\le k}\rho(X_i)>ur\}$ according to the smallest index $i$ such that
$\rho(X_i)>ur$. By stationarity, we find
\begin{align*} \allowdisplaybreaks
\lefteqn{
\frac{1}{V(u)}\E[f(X_1/u,\ldots,X_k/u)]
} \\
  &= 
  \frac{1}{V(u)} 
  \E\bigg[ 
    f(X_1/u,\ldots,X_k/u) \, 
    \1\bigg( \max_{1\le i \le k} \rho(X_i/u) > r \bigg) 
  \bigg] \\
&= \frac{1}{V(u)} \sum_{i=1}^k\E\bigg[f(X_1/u,\ldots,X_k/u)\1\bigg(
 \rho(X_i)>ur\ge \max_{1\le j\le i-1}\rho(X_j) \bigg) \bigg]\\ 
&= \frac{V(ur)}{V(u)} \sum_{i=1}^k\E \bigg[f(X_1/u,\ldots,X_k/u)\1\bigg(
\max_{1\le j\le i-1}\rho(X_j)\le ur \bigg) \bigg|\,\rho(X_i)>ur \bigg]\\ 
&=\frac{V(ur)}{V(u)} \sum_{i=1}^k
  \E\bigg[f(X_{1-i}/u,\ldots,X_{k-i}/u)\1 \bigg(\max_{1-i\le j\le
  -1}\rho(X_j)<ur \bigg) \bigg|\,\rho(X_0)>ur \bigg] \,.
\end{align*}
Writing
\begin{equation*}
  \frac{X_t}{u}=r \, \frac{\rho(X_0)}{ur} \,\frac{X_t}{\rho(X_0)}\,,
\end{equation*}
we have, by \eqref{eq:spp02} and by continuity of the $\Pareto(\alpha)$ distribution,
\begin{align}
\label{eq:limitr}
\lefteqn{\lim_{u\to \infty}\frac{1}{V(u)}\E[f(X_1/u,\ldots,X_k/u)]}\\
&=
\nonumber
  r^{-\alpha} 
  \sum_{i=1}^k 
  \int_1^{\infty}
    \E \bigg[
      f(ry\Theta_{1-i},\ldots,ry\Theta_{k-i}) \,
      \1 \bigg(\max_{1-i \le j\le -1} \rho(y\Theta_j)<1 \bigg) 
    \bigg] \,
  d(-y^{-\alpha}) \\
&= 
\nonumber
  \sum_{i=1}^k 
  \int_{r}^{\infty}
    \E \biggl[
      f(z \Theta_{1-i},\ldots,z \Theta_{k-i}) \,
      \1 \biggl( \max_{1-i \le j \le -1} \rho(z\Theta_j) < r \biggr)
    \biggr] \,
  d(-z^{-\alpha}),
\end{align}
where we substitued $z = ry$. 
The final expression involves an arbitrary scalar $r \in (0, r_0)$ but, in view of the left-hand side of \eqref{eq:limitr}, it does not depend on the exact value of $r$. We show that we can take the limit as $r \downarrow 0$, obtaining~\eqref{eq:intfnuk}. To that end, we apply dominated convergence to each term $i \in \{1, \ldots, k\}$ separately. For fixed $z \in (0, \infty)$, we have, since $f$ is bounded,
\begin{multline*}
  \lim_{r \downarrow 0}
  \E \biggl[
    f(z \Theta_{1-i},\ldots,z \Theta_{k-i}) \,
    \1 \biggl( \max_{1-i \le j \le -1} \rho(z\Theta_j) < r \biggr)
  \biggr]  \\
  =
    \E \biggl[ 
      f(0,\ldots,0,z\Theta_0,\ldots,z\Theta_{k-i}) \, 
      \1 \biggl( \max_{1-i\le j\le -1}\rho(\Theta_j)=0 \biggr) 
    \biggr].
\end{multline*}
Next, we need to show that we can integrate this limit over $z \in (0, \infty)$ according to the measure $d(-z^{-\alpha})$. Since $f$ is bounded and vanishes on the set $\{ \bm{x} \in S^k : \max_{1 \le i \le k} \rho(x_i) \le r_0 \}$, there exists $c > 0$ such that
\[
  \abs{ f(\bm{x}) }
  \le
  c \, \1 \left( \max_{1 \le j \le k} \rho(x_j) \ge r_0 \right),
  \qquad
  \bm{x} \in S^k \setminus \{(0, \ldots, 0)\}.
\]
It follows that, for all $z \in (0, \infty)$ and all $r \in (0, r_0)$,
\begin{multline*}
  \lefteqn{
  \left\lvert \E \biggl[
    f(z \Theta_{1-i},\ldots,z \Theta_{k-i}) \,
    \1 \biggl( \max_{1-i \le j \le -1} \rho(z\Theta_j) < r \biggr)
  \biggr] \right\rvert 
  } \\
  \le
  c \, \Pr \left[ \max_{1 \le j \le k} \rho(z \Theta_{j-i}) \ge r_0 \right]
  \le
  c \, \sum_{j=1}^k \Pr[ \rho( \Theta_{j-i}) \ge z^{-1} r_0 ].
\end{multline*}
For any integer $t$, we have, by Fubini's theorem,
\[
  \int_0^\infty \Pr[ \rho( \Theta_t ) \ge z^{-1} r_0 ] \, d(-z^{-\alpha})
  = r_0^{-\alpha} \E[ \rho( \Theta_t)^\alpha ]
  \le r_0^{-\alpha},
\]
the inequality following from \eqref{eq:claim1}. This justifies the use of the dominated convergence theorem when passing to the limit $r \downarrow 0$ on the right-hand side of \eqref{eq:limitr}. We arrive at \eqref{eq:spp03} with limit measure $\nu_k$ as given in \eqref{eq:intfnuk}. This completes the proof of Theorem~\ref{thm:spectralprocess}.
\end{proof}

\section{The time-change formula}
\label{sec:time-change-formula}

In general, the spectral process $(\Theta_t)_{t \in \ZZ}$ of a
stationary regularly varying time series $(X_t)_{t \in \ZZ}$ is itself
nonstationary. Still, the fact that $(X_t)_{t \in \ZZ}$ is stationary
induces a peculiar structure on the distribution of the spectral
process. In particular, the distribution of $(\Theta_t)_{t \in \ZZ}$
is determined by the distribution of its restriction to the
nonnegative time axis, that is, of the \emph{forward spectral process}
$(\Theta_t)_{t \in \ZZ_+}$, with $\ZZ_+ = \{0, 1, 2, \ldots \}$. The
same is true for the \emph{backward spectral process} $(\Theta_t)_{t
\in \ZZ_-}$, with $\ZZ_- = \{0, -1, -2, \ldots \}$.

\begin{theorem}\label{thm:timechange}
Statements (ii) and (iii) in Theorem \ref{thm:spectralprocess} are
equivalent to the statements with $\mathbb{Z}$ replaced by
$\mathbb{Z}_+$ or $\mathbb{Z}_-$. In that case,
\begin{equation}
  \label{eq:timechange}
  \E[f(\Theta_{-s},\ldots,\Theta_t)]
  =
  \E\bigg[
    f \bigg(\frac{\Theta_0}{\rho(\Theta_s)},\ldots,\frac{\Theta_{t+s}}{\rho(\Theta_s)} \bigg) \,
    \rho(\Theta_s)^{\alpha}
  \bigg]  
\end{equation}
for all nonnegative integers $s$ and $t$ and for all integrable functions $f:S^{t+s+1}\to \mathbb{R}$ with the property that $f(\theta_{-s},\ldots,\theta_t)=0$ whenever $\theta_{-s}=0$.   
\end{theorem}

By `integrable functions' is meant real-valued, Borel-measurable functions such that one of the expectations, and hence the other one, exists.
In \eqref{eq:timechange} and in later formulas in which expressions like $\rho(\Theta_s)$ appear both in the denominator and as a term in a product, the integrand is to be interpreted as zero when $\rho(\Theta_s)$ is zero. A time-change formula for general integrable functions, without the zero-property, is given in \eqref{eq:timechange:noZero} inside the proof of Theorem~\ref{thm:timechange}.

By considering the time-reversed process $\widetilde{X}_t =
 X_{-t}$, equation~\eqref{eq:timechange} can be reversed in the
 obvious way. A simple case occurs when $f$ only depends on its first
 component, that is, when $f(\theta_{-s}, \ldots, \theta_t) \equiv
 f(\theta_{-s})$ and $f(0) = 0$: equation~\eqref{eq:timechange} then reduces
 to 
\begin{equation}
\label{E:time:s}
  \E [ f(\Theta_{-s}) ] = \E [ f ( \Theta_0 / \rho(\Theta_s) ) \,
  \rho(\Theta_s)^\alpha ], \qquad s \in \ZZ. 
\end{equation}
This yields an expression of the distribution of $\Theta_{-s}$ in terms of the joint law of $\Theta_0$ and $\Theta_s$. In particular, we find
\[
  \Pr[ \Theta_{-s} \ne 0 ] = \E [ \rho(\Theta_s)^\alpha ], \qquad s \in \ZZ.
\]
If the common value in the preceding display is equal to unity, then
\eqref{E:time:s} is valid for arbitrary integrable $f$, that is,
without the restriction that $f(0) = 0$.  

\begin{proof}[Proof of Theorem~\ref{thm:timechange}]
By symmetry, we only need to consider the forward case, $\mathbb{Z}_+=\{0,1,2,\ldots\}$. Consider the statements (ii) and (iii) in Theorem~\ref{thm:spectralprocess} with $\mathbb{Z}$ replaced by $\mathbb{Z}_+$. 
\begin{itemize}
\item[(ii$_+$)] 
  The function $u \mapsto \Pr[\rho(X_0) > u]$ belongs to $\RV_{-\alpha}$, and in $S^{\ZZ_+}$,
  \[ 
    \law\biggl( 
      (X_t / \rho(X_0))_{t \in \ZZ_+} \, \big| \,  \rho(X_0) > u
    \biggr) 
    \dto (\Theta_t)_{t \in \ZZ_+} \qquad (u \to \infty). 
  \]

\item[(iii$_+$)] 
  In $(0, \infty) \times S^{\ZZ_+}$, as $u \to \infty$,
  \[
    \law \biggl(  
      \rho(X_0) / u, (X_t / \rho(X_0))_{t \in \ZZ_+} \, \big| \,  \rho(X_0) > u \biggr) 
    \dto \bigl( Y, (\Theta_t)_{t \in \ZZ_+} \bigr),
  \] 
  where $Y$ is a $\Pareto(\alpha)$ random variable independent from $(\Theta_t)_{t \in \ZZ_+}$.
\end{itemize}
We have to show that the statements (i)--(iii) in Theorem~\ref{thm:spectralprocess} are equivalent with each of (ii$_+$) and (iii$_+$). We already know that (i) implies (ii). Trivially, (ii) implies (ii$_+$). To show that (ii$_+$) implies (iii$_+$), just set $s = 0$ in the part of the proof of Theorem~\ref{thm:spectralprocess} that (ii) implies (iii). Since (iii) implies (i) by Theorem~\ref{thm:spectralprocess}, all that remains to be shown is that (iii$_+$) implies (iii). 

The proof of \eqref{eq:claim1} in Theorem~\ref{thm:spectralprocess}
ensures that if (iii$_+$), then for every $t \in \ZZ_+$, \eqref{eq:claim1} holds.

\begin{lemma}\label{CL:minust}
  If (iii$_+$), then for every $t \in \mathbb{Z}_+$,
  \begin{equation*}
    \law \left( X_{-t}/\rho(X_0) \mid \rho(X_0) > u \right) \dto \nu_t\,, \quad u\to \infty\,,
  \end{equation*}
where $\nu_t$ is a probability measure on $S$ given for
$\nu_t$-integrable $g:S\to \mathbb{R}$ by 
\begin{equation*}
  \int g\,d\nu_t
  =
  g(0) \left\{ 1-\E[\rho(X_0)^{\alpha}] \right\}
  +
  \E[ g(\Theta_0/\rho(\Theta_t)) \, \rho(\Theta_t)^{\alpha}]\,.
\end{equation*}
\end{lemma}

\begin{proof}[Proof of Lemma~\ref{CL:minust}]
  Let $g:S\to \mathbb{R}$ be continuous and bounded. Fix $r>0$. We
  have
\begin{align*}
&\E [g(X_{-t}/\rho(X_0))\mid \rho(X_0)>u]\\
&= g(0)\Pr [\rho(X_{-t})\le ru\mid \rho(X_0)>u]\\
&\quad + \E [\{g(X_{-t}/\rho(X_0))-g(0)\} \, \1 (\rho(X_{-t})\le ru)\mid \rho(X_0)>u]\\ 
&\quad +\E [g(X_{-t}/\rho(X_0)) \, \1 (\rho(X_{-t})>ru)\mid \rho(X_0)>u]\\ 
&=Q_1+Q_2+Q_3 \,.
\end{align*}
The first term $Q_1$ on the right-hand side has been treated in
\eqref{eq:claim1}. If $\rho(X_0)>u$ and $\rho(X_{-t})\le ru$, then
$\rho(X_{-t}/\rho(X_0))=\rho(X_{-t})/\rho(X_0)<r$. Recall that there
exist positive scalars $(z_r)_{r>0}$ such that 
  $\{x:\rho(x)<r\}\subset \{ x:d(x,0)<z_r\}$ and $\lim_{r\downarrow 
    0}z_r=0$. Since $g$ is continuous,
\begin{equation*}
  \lim_{r\downarrow 0}\limsup_{u\to \infty} |Q_2|
  \le \lim_{r\downarrow 0} \sup_{x : \rho(x) < r} |g(x)-g(0)|
  = 0\,.
\end{equation*}
For $Q_3$, writing $V(u)=\Pr[\rho(X_0)>u]$, we have, by stationarity
of $\bm{X}$,
\begin{align*}
Q_3
&= 
  \frac{V(ru)}{V(u)} 
  \E[g(X_0/\rho(X_t)) \, \1 (\rho(X_t)>u) \mid \rho(X_0)>ru]\\ 
&= 
  \frac{V(ru)}{V(u)} 
  \E \bigg[
  g\bigg(\frac{X_0/\rho(X_0)}{\rho(X_t)/\rho(X_0)} \bigg) \, 
  \1 \bigg(r \, \frac{\rho(X_0)}{ru} \, \frac{\rho(X_t)}{\rho(X_0)} >1 \bigg) 
  \,\bigg|\,
  \rho(X_0) > ru 
  \bigg] \\ 
&\to 
  r^{-\alpha} \E [g(\Theta_0/\rho(\Theta_t)) \, \1 (rY\rho(\Theta_t)>1)]\,,
  \qquad u\to \infty\,.
\end{align*}
The last step is justified by (iii$_+$), which implies the continuity of the
law of $Y$ and the independence between $Y$ and $\Theta_t$. Moreover,
this limit relation holds for every $r > 0$ in a neighborhood of zero. The
limit is equal to $\E[g(\Theta_0/\rho(\Theta_t))\min
\{\rho(\Theta_t),r^{-1}\}^{\alpha}]$, which, by dominated convergence,
tends to $\E [g(\Theta_0/\rho(\Theta_t)) \, \rho(\Theta_t)^{\alpha}]$ as
$r\downarrow 0$. Therefore, Lemma~\ref{CL:minust} is established. 
\end{proof}

Fix nonnegative integer $s$ and $t$. If (iii$_+$), then in view of
Lemma~\ref{CL:minust}, the converse half of Prohorov's theorem
\citep[Theorem~5.2]{Billingsley:1999uc} and Tychonoff's theorem, there exists $u_0 > 0$ such that the collection of probability measures
\begin{equation}
\label{E:laws:u}
  \mathcal{L}\bigl( X_{-s} / \rho(X_0), \ldots, X_t /\rho(X_0) \mid \rho(X_0) > u \bigr), \qquad u > u_0,
\end{equation}
is \emph{tight}, that is, for every $\eps > 0$ there exists a compact
subset $K_\eps$ of $S^{t+s+1}$ so that the probability mass of
$K_\eps$ under each of the laws above is at least $1 - \eps$. By the
direct half of Prohorov's theorem \citep[Theorem~5.1]{Billingsley:1999uc}, the collection of probability
measures above is \emph{relatively compact}: for every sequence $u_n
\to \infty$ there exists a subsequence $u_{n_m} \to \infty$ for which
the laws have a limit in distribution. To prove convergence in
distribution of \eqref{E:laws:u} as $u \to \infty$, it is then
sufficient to show uniqueness of the possible sequential limits. As
probability distributions are determined by their integrals of
bounded, Lipschitz continuous functions \citep[proof of Theorem~1.2]{Billingsley:1999uc}, it is sufficient to show the 
following lemma. 

\begin{lemma}
\label{CL:uniqueness}
If (iii$_+$), then for every nonnegative integer $s$ and $t$ and for every bounded, Lipschitz continuous function $f : S^{t+s+1} \to \mathbb{R}$, the following limit exists:
\begin{equation}
\label{E:limit2exist}
  \lim_{u \to \infty} \E [ f( X_{-s}/ \rho(X_0), \ldots, X_t / \rho(X_0)) \mid \rho(X_0) > u].
\end{equation}
\bgroup
If moreover $f(\theta_{-s}, \ldots, \theta_t) = 0$ as soon as $\theta_{-s} = 0$, then the limit is equal to
\begin{equation}
\label{E:TimeChange}
  \E \biggl[
    f \biggl( \frac{\Theta_0}{\rho(\Theta_s)},\ldots,\frac{\Theta_{t+s}}{\rho(\Theta_s)} \biggr) \, 
    \rho(\Theta_s)^{\alpha} 
  \biggr].
\end{equation}
\egroup
\end{lemma}

\begin{proof}[Proof of Lemma~\ref{CL:uniqueness}]
We fix an integer $t \ge 0$ and proceed by induction on the integer $s \ge 0$. The case $s=0$ is already included in (ii$_+$) or (iii$_+$). Note that $\rho(\Theta_0) = 1$ with probability one.

Let $s\ge 1$ be an integer and assume the stated convergence holds for $s$ replaced by $s-1$, and all bounded, Lipschitz continuous functions from $S^{t+(s-1)+1}$ into $\mathbb{R}$. Let $f:S^{t+s+1}\to \mathbb{R}$ be bounded and Lipschitz continuous with Lipschitz constant $L>0$. Define $f_0:S^{t+s+1}\to \mathbb{R}$ by
\begin{equation}
\label{eq:f0}
  f_0(\theta_{-s},\ldots,\theta_t)=f(\theta_{-s},\ldots,\theta_t)-f(0,\theta_{-s+1},\ldots,\theta_t)\,.
\end{equation}
We have
\begin{align*}
&\E [f(X_{-s}/\rho(X_0),\ldots, X_t/\rho(X_0))\mid \rho(X_0)>u]\\
&\quad =\E [f(0,X_{-s+1}/\rho(X_0),\ldots,X_t/\rho(X_0))\mid \rho(X_0)>u]\\
&\qquad + \E [f_0(X_{-s}/\rho(X_0),\ldots,X_t/\rho(X_0))\mid \rho(X_0)>u].
\end{align*}
By the induction hypothesis, the following limit already exists:
\begin{equation*}
  \lim_{u\to \infty}\E [f(0,X_{-s+1}/\rho(X_0),\ldots,X_t/\rho(X_0))\mid \rho(X_0)>u].
\end{equation*}
\bgroup
We will show that
\begin{multline}
\label{eq:f0:limit}
  \lim_{u \to \infty}
  \E \biggl[ f_0 \biggl( \frac{X_{-s}}{\rho(X_0)}, \ldots, \frac{X_t}{\rho(X_0)} \biggr) \, \bigg| \, \rho(X_0)>u \biggr] \\
  = 
  \E \bigg[
    f_0 \biggl( \frac{\Theta_0}{\rho(\Theta_s)},\ldots,\frac{\Theta_{t+s}}{\rho(\Theta_s)} \biggr) \, 
    \rho(\Theta_s)^{\alpha} 
  \bigg].
\end{multline}
Fix $r > 0$ and split the integrand on the left-hand side into two parts, according to whether $\rho(X_{-s}) \le ru$ or $\rho(X_{-s}) > ru$. By the triangle inequality, equation~\eqref{eq:f0:limit} will be the consequence of the following three limits:
\begin{equation}
\label{eq:f0:limit:1}
  \lim_{r\downarrow 0} 
  \limsup_{u \to \infty}
  \E \biggl[ 
    \biggl\vert f_0 \biggl( \frac{X_{-s}}{\rho(X_0)}, \ldots, \frac{X_t}{\rho(X_0)} \biggr) \biggr\vert \, 
    \1\{\rho(X_{-s}) \le ru\} \, \bigg| \, 
    \rho(X_0)>u 
  \biggr] 
  = 0,
\end{equation}
\begin{multline}
\label{eq:f0:limit:2}
  \lim_{u \to \infty}
  \E \biggl[
    f_0 \biggl( \frac{X_{-s}}{\rho(X_0)}, \ldots, \frac{X_t}{\rho(X_0)} \biggr) \, 
    \1 \{\rho(X_{-s}) > ru\} \, \bigg| \, 
    \rho(X_0)>u 
  \biggr] \\
  = 
  \E \biggl[
    f_0 \bigg(\frac{\Theta_0}{\rho(\Theta_s)},\ldots,\frac{\Theta_{t+s}}{\rho(\Theta_s)} \bigg) 
    \min \{\rho(\Theta_s),r^{-1}\}^{\alpha}
  \biggr], 
\end{multline}
\begin{multline}
\label{eq:f0:limit:3}
  \lim_{r \downarrow 0}
  \E \biggl[
    f_0 \biggl(\frac{\Theta_0}{\rho(\Theta_s)},\ldots,\frac{\Theta_{t+s}}{\rho(\Theta_s)} \biggr) 
    \min \{\rho(\Theta_s),r^{-1}\}^{\alpha}
  \biggr] \\
  =
  \E \biggl[
    f_0 \biggl( \frac{\Theta_0}{\rho(\Theta_s)},\ldots,\frac{\Theta_{t+s}}{\rho(\Theta_s)} \biggr) \, 
    \rho(\Theta_s)^{\alpha} 
  \biggr].  
\end{multline}
We will show equations~\eqref{eq:f0:limit:1}, \eqref{eq:f0:limit:2}, and \eqref{eq:f0:limit:3}.

First we show \eqref{eq:f0:limit:1}. Recall that there exist positive scalars $(z_r)_{r>0}$ such that $\{x:\rho(x)<r\}\subset \{ x:d(x,0)<z_r\}$ for every $r>0$ and $\lim_{r \downarrow 0} z_r = 0$. By definition of $f_0$ in \eqref{eq:f0} and the fact that $f$ is Lipschitz continuous with some constant $L > 0$, we find that the expectation on the left-hand side in \eqref{eq:f0:limit:1} is bounded by $L z_r$. This converges to zero as $r \downarrow 0$.
\egroup


Next we show \eqref{eq:f0:limit:2}. Let $V(u)=\Pr[\rho(X_0)>u]$. By stationarity of $(X_t)_{t \in \mathbb{Z}}$, regular variation of $V$, and (iii$_+$), we have
\begin{align*}
&\E [f_0(X_{-s}/\rho(X_0),\ldots,X_t/\rho(X_0)) \, \1 (\rho(X_{-s})>ru)\mid
  \rho(X_0)>u]\\
&= \frac{V(ru)}{V(u)} \E
  [f_0(X_0/\rho(X_s),\ldots,X_{t+s}/\rho(X_s)) \, \1(\rho(X_s)>u)\mid
  \rho(X_0)>ru]\\
&\to r^{-\alpha}
  \E \bigg[f_0
  \bigg(\frac{\Theta_0}{\rho(\Theta_s)},\ldots,\frac{\Theta_{t+s}}{\rho(\Theta_s)}
  \bigg)\1 (rY\rho(\Theta_s)>1) \bigg]\,, \qquad u\to \infty\,.
\end{align*}
\bgroup
The passage to the limit is justified by (iii$_+$), the continuity of $Y$, and the independence of $Y$ and $(\Theta_t)_{t \in \mathbb{Z}_+}$. By Fubini's theorem, the expression on the right-hand side is equal to
\begin{align*}
  \lefteqn{
  r^{-\alpha}
  \int_1^\infty 
  \E \biggl[
    f_0 \biggl( \frac{\Theta_0}{\rho(\Theta_s)},\ldots,\frac{\Theta_{t+s}}{\rho(\Theta_s)} \biggr) \,
    \1 \{ rz\rho(\Theta_s)>1 \}
  \biggr]
  d(-z^{-\alpha})
  } \\
  &=
  \int_r^\infty
  \E \biggl[
    f_0 \biggl( \frac{\Theta_0}{\rho(\Theta_s)},\ldots,\frac{\Theta_{t+s}}{\rho(\Theta_s)} \biggr) \,
    \1 \{z\rho(\Theta_s)>1\}
  \biggr]
  d(-z^{-\alpha}) \\
  &=
  \E \biggl[
    f_0 \biggl( \frac{\Theta_0}{\rho(\Theta_s)},\ldots,\frac{\Theta_{t+s}}{\rho(\Theta_s)} \biggr) \,
    \int_r^\infty \1 \{z\rho(\Theta_s)>1\} \, d(-z^{-\alpha})
  \biggr] \\
  &=
  \E \biggl[
    f_0 \biggl( \frac{\Theta_0}{\rho(\Theta_s)},\ldots,\frac{\Theta_{t+s}}{\rho(\Theta_s)} \biggr) \,
    \min \{ \rho(\Theta_s), r^{-1} \}^{\alpha}
  \biggr]. 
\end{align*}
We arrive at \eqref{eq:f0:limit:2}.

Finally, the proof of \eqref{eq:f0:limit:3} is immediate in view of the dominated convergence theorem, the boundedness of $f$, and the integrability of $\rho(\Theta_s)^{\alpha}$, see \eqref{eq:claim1}.

We have now proven \eqref{eq:f0:limit:1}, \eqref{eq:f0:limit:2} and \eqref{eq:f0:limit:3} and thus \eqref{eq:f0:limit}. If the function $f$ is such that $f(\theta_{-s}, \ldots, \theta_t) = 0$ as soon as $\theta_{-s} = 0$, then $f = f_0$ and \eqref{E:TimeChange} follows. This finishes the proof of Lemma~\ref{CL:uniqueness}.
\egroup
\end{proof}

By Lemma~\ref{CL:uniqueness} and the tightness argument preceding it, condition~(iii$_+$) implies that the limit in distribution
\[
  \mathcal{L}\bigl( X_{-s} / \rho(X_0), \ldots, X_t / \rho(X_0) \mid \rho(X_0) > u \bigr)\,, \qquad u \to \infty\,,
\]
exists for all nonnegative integer $s$ and $t$. By the Daniell--Kolmogorov extension theorem \citep[Chapter~4, Theorem~53]{POL02}, these limits in distributions are the `finite-dimensional' distributions of a random element $(\Theta_t)_{t \in \ZZ}$ in the product space $S^\ZZ$. Statement (iii) concerning weak convergence in $S^\ZZ$ then follows from the convergence in the previous display for all $s$ and $t$ together with \citet[Theorem~1.4.8]{VDV96}.

It remains to show equation~\eqref{eq:timechange}. The weak convergence established in the previous paragraph together with Lemma~\ref{CL:uniqueness} imply that for bounded, Lipschitz continuous functions $f : S^{t+s+1} \to \reals$ vanishing on $\{ (\theta_{-s}, \ldots, \theta_t) \in S^{t+s+1} : \theta_{-s} = 0 \}$, we have
\begin{multline}
\label{E:TimeChange:2}
  \E[f(\Theta_{-s},\ldots,\Theta_t)]
  =
  \lim_{u \to \infty}
  \E \biggl[
    f \biggl( \frac{X_{-s}}{\rho(X_0)}, \ldots, \frac{X_t}{\rho(X_0)} \biggr) \, \bigg| \, 
    \rho(X_0)>u 
  \biggr] \\
  = 
  \E\biggl[
    f \bigg(\frac{\Theta_0}{\rho(\Theta_s)},\ldots,\frac{\Theta_{t+s}}{\rho(\Theta_s)} \bigg) \,
    \rho(\Theta_s)^{\alpha}
  \biggr].   
\end{multline}
Let $g : S^{t+s+1} \to \mathbb{R}$ be bounded and Lipschitz continuous. Write $g(\Theta_{-s}, \ldots, \Theta_t)$ as a telescoping sum of $s+1$ terms:
\begin{align*}
  g(\Theta_{-s}, \ldots, \Theta_t) 
  &= g(\Theta_{-s}, \ldots, \Theta_t) - g(0, \Theta_{-s+1}, \ldots, \Theta_t) \\
  &\quad \mbox{} + g(0, \Theta_{-s+1}, \ldots, \Theta_t) - g(0, 0, \Theta_{-s+2}, \ldots, \Theta_t) \\
  &\quad \mbox{} + \cdots \\
  &\quad \mbox{} + g(0, \ldots, 0, \Theta_{-1}, \ldots, \Theta_t) - g(0, \ldots, 0, \Theta_0, \ldots, \Theta_t) \\
  &\quad \mbox{} + g(0, \ldots, 0, \Theta_0, \ldots, \Theta_t).
\end{align*}
Take expectations on both sides and apply \eqref{E:TimeChange:2} to the first $s$ lines of the right-hand side of the previous display at $s$ replaced by $s, s-1, \ldots, 1$, respectively, to obtain
\begin{align}
\label{eq:timechange:noZero}
  \lefteqn{\E [ g(\Theta_{-s}, \ldots, \Theta_t) ]} \\
\nonumber
  &= \textstyle 
  \E \biggl[ \biggl\{ g \biggl( \frac{\Theta_0}{\rho(\Theta_s)}, \ldots, \frac{\Theta_{t+s}}{\rho(\Theta_s)} \biggr) 
    - g \biggl( 0, \frac{\Theta_1}{\rho(\Theta_s)}, \ldots, \frac{\Theta_{t+s}}{\rho(\Theta_s)} \biggr) \biggr\} \, \rho(\Theta_s)^\alpha \biggr] \\
\nonumber
  &\quad \textstyle \mbox{} + \E \biggl[ \biggl\{ g \biggl( 0, \frac{\Theta_0}{\rho(\Theta_{s-1})}, \ldots, \frac{\Theta_{t+s-1}}{\rho(\Theta_{s-1})} \biggr) 
    - g \biggl( 0, 0, \frac{\Theta_1}{\rho(\Theta_{s-1})}, \ldots, \frac{\Theta_{t+s-1}}{\rho(\Theta_{s-1})} \biggr) \biggr\} \, \rho(\Theta_{s-1})^\alpha \biggr] \\
\nonumber
  &\quad \mbox{} + \cdots \\
\nonumber
  &\quad \textstyle \mbox{} + \E \biggl[ \biggl\{ g \biggl( 0, \ldots, 0, \frac{\Theta_0}{\rho(\Theta_{1})}, \ldots, \frac{\Theta_{t+1}}{\rho(\Theta_{1})} \biggr) 
    - g \biggl( 0, \ldots, 0, \frac{\Theta_1}{\rho(\Theta_{1})}, \ldots, \frac{\Theta_{t+1}}{\rho(\Theta_{1})} \biggr) \biggr\} \, \rho(\Theta_{1})^\alpha \biggr] \\
\nonumber
  &\quad \textstyle \mbox{} + \E [ g(0, \ldots, 0, \Theta_0, \ldots, \Theta_t) ].
\end{align}
The equality in the preceding display being true for all bounded and Lipschitz continuous functions $g : S^{t+s+1} \to \mathbb{R}$, it must hold whenever $g$ is the indicator function of a closed set \citep[proof of Theorem~1.2]{Billingsley:1999uc} and then, by a standard argument, also for all measurable functions $S^{t+s+1} \to \mathbb{R}$ that are integrable with respect to the law of $(\Theta_{-s}, \ldots, \Theta_t)$. For such functions that vanish whenever their first argument is equal to zero, the formula in the preceding display simplifies to~\eqref{eq:timechange} again.

This concludes the proof of Theorem~\ref{thm:timechange}.
\end{proof}

\section{Discussion}
\label{sec:disc}

On $S = [0, \infty)^2$, the function $\rho(x, y) = \min(x, y)$ is \emph{not} a modulus, since condition~(iii) in Definition~\ref{def:radius} is not satisfied. 
Similarly, \citet{dombry+r:2015} consider `cost functionals' that satisfy conditions (i) and (ii) but not necessarily (iii) in Definition~\ref{def:radius}. Without the latter condition, however, regular variation in $S$ can no longer be characterized via a polar decomposition as in Proposition~\ref{prop:spec00}, since sets of the form $\{ x : \rho(x) < r \}$ do no longer form a neighbourhood base of the origin.

Relative to such `pseudo-moduli', \emph{hidden regular variation} \citep{RES02} on subcones may still occur. The notion of $M_0$-convergence then needs to be replaced by a more general one, involving sets that are bounded away from some `forbidden set' which may be larger than a singleton. In $S = [0, \infty)^2$, one could for instance exclude the union of the two coordinate axes. Such a concept of regular variation is relevant for stochastic volatility models, for example, which exhibit asymptotic independence and therefore trivial spectral tail processes in the sense of this paper \citep{MZ15, JD16}. A complication, however, is that the index of hidden regular variation may depend on the time lag \citep{kulik+s:2015}. A general treatment for such time series in metric spaces is an interesting research problem.

\appendix
\section{Convergence of measures}
\label{sec:m_0-convergence}

We consider a complete, separable metric space $(S,d)$ and some point $0 \in S$. For $A \subset S$, let $A^{\circ}$ and $A^-$ denote the interior and closure of $A$, respectively, and let $\partial A=A^-\setminus A^{\circ}$ be the boundary of $A$. Recall $B_{0, u} = \{ x \in S : d(x, 0) < u \}$ for $u > 0$ as well as the space $M_0(S)$ from Section~\ref{sec:rv}. Let $M_b(\mathcal{X})$ denote the set of finite Borel measures on some metric space $\mathcal{X}$ and define convergence of measures in $M_b(\mathcal{X})$ by the usual notion of weak convergence, i.e., convergence of integrals of bounded, continuous functions from $\mathcal{X}$ into $\RR$. We begin with a variation on Theorem~2.2 in \cite{HL06}.

\begin{lemma}
\label{lem:hh1}
\begin{enumerate}[(i)]
\item
Assume $\mu_n \to \mu$ in $M_0(S)$ as $n \to \infty$ and let $f : S_0 \to \RR$ be bounded, measurable, and vanish on $B_{0,u}$ for some $u > 0$. Let $D$ be the discontinuity set of $f$. If $\mu(D) = 0$, then $\int f \, d\mu_n \to \int f \, d\mu$ as $n \to \infty$. 
\item
If there exists a decreasing sequence of positive scalars $(r_i)_{i \in \NN}$ with $r_i\to 0$ as $i\to \infty$ such that for each $i$, there   exists a neighbourhood of the origin $0$, say $N_i$, such that $N_i\subset B_{0,r_i}$ and $\mu_n(\,\cdot\,\setminus N_i)\to \mu(\,\cdot\,\setminus N_i)$ in $M_b(S \setminus N_i)$, then $\mu_n \to \mu$ in $M_0(S)$ as $n \to \infty$. 
\end{enumerate}
\end{lemma}

\begin{proof}
(i) Let $r \in (0, u)$ be such that $\mu(\boundary B_{0,r}) = 0$. Let $\mu_n^{(r)}$ and $\mu_n^{(r)}$ denote the restrictions of $\mu_n$ and $\mu$ to $S \setminus B_{0,r}$, respectively. By (the proof of) Theorem~2.2 in \cite{HL06}, we have weak convergence $\mu_n^{(r)} \to \mu^{(r)}$ as $n \to \infty$ in the space $M_b(S \setminus B_{0,r})$. By the continuous mapping theorem for weak convergence of finite measures, $\int_{S_0} f \, d\mu_n = \int_{S \setminus B_{0,r}} f \, d\mu_n^{(r)} \to \int_{S \setminus B_{0,r}} f \, d\mu^{(r)} = \int_{S_0} f \, d\mu$ as $n \to \infty$. 

(ii) For any $f\in \mathcal{C}_0$, there exists $i \in \NN$ such that $f$ vanishes on $B_{0,r_i}$ and consequently on $N_i$. Since $\mu_n(\point \setminus N_i)\to \mu(\point \setminus N_i)$ in $M_b(S \setminus N_i)$, we have $\int_{S_0} f\, d\mu_n = \int_{S \setminus N_i}f\,d\mu_n \to \int_{S \setminus N_i} f \, d\mu = \int_{S_0} f \, d\mu$. Therefore, $\mu_n \to \mu$ in $M_0(S)$ as $n \to \infty$.
\end{proof}

The following lemma is useful for proving convergence in distribution. 

\begin{lemma}
\label{L:product}
Let $(S,d)$ be a separable metric space. Let $(X_n, Y_n)$ and $(X, Y)$ be random elements in $\RR \times S$. Then $(X_n, Y_n) \dto (X, Y)$ if and only if
\begin{equation}
\label{E:product}
  \E [ \1(X_n \le x) \, g(Y_n) ] \to \E [ \1(X \le x) \, g(Y) ] \qquad (n \to \infty)
\end{equation}
for every continuity point $x \in \RR$ of $X$ and every bounded and continuous function $g : S \to \RR$.
\end{lemma}

\begin{proof}
The `only if' part is a special case of the continuous mapping theorem. So assume \eqref{E:product} holds. Taking $g \equiv 1$ yields $X_n \dto X$. Taking $x$ arbitrarily large so that $\Pr[X > x]$ is arbitrarily small yields $Y_n \dto Y$. As a consequence, the sequence $(X_n, Y_n)$ is tight. It remains to show that the joint distribution of $(X, Y)$ is determined by expectations as in the right-hand side \eqref{E:product}. By Lemma~1.4.2 in \citet{VDV96}, the joint distribution of $(X, Y)$ is determined by expectations of the form $\E [ f(X) \, g(Y) ]$ with $f : \RR \to \RR$ and $g : S \to \RR$ nonnegative, Lipschitz continuous, and bounded. It then suffices to write $f$ as the limit of an increasing sequence of step functions whose jump locations are continuity points of $X$.
\end{proof}

The following theorem is similar to Theorems~2.2 and~2.3 in \cite{Billingsley:1999uc} and provides a criterion for convergence in $M_0(S)$.

\begin{theorem}
\label{thm:hh1}
Suppose that $\mathcal{A}$ is a $\pi$-system on $S$ satisfying the following two conditions:
\begin{enumerate}[({C}1)]
\item
There exists a decreasing sequence $(r_i)_{i \in \NN}$ of positive scalars with $r_i \to 0$ as $i\to \infty$ such that for each $i$, there exists a neighbourhood of the point $0$, say $N_i$, such that $N_i \subset B_{0,r_i}$ and $A \setminus N_i \in \mathcal{A}$ for all $A \in \mathcal{A}$. 
\item
Each open subset $G$ of $S$ with $0 \notin G^-$ is a countable union of $\mathcal{A}$-sets.
\end{enumerate}
If $\mu_n(A) \to \mu(A)$ as $n \to \infty$, for all $A$ in $\mathcal{A}$, then $\mu_n \to \mu$ in $M_0(S)$ as $n \to \infty$.
\end{theorem}

\begin{proof}
Let $i \in \NN$; by Lemma~\ref{lem:hh1}, it is sufficient to show that $\mu_n( \point \setminus N_i ) \to \mu( \point \setminus N_i ) $ in $M_b(S \setminus N_i)$ as $n \to \infty$. To do so, we apply the Portmanteau theorem for weak convergence of finite measures \citep[Theorem~2.1]{Billingsley:1999uc}. Any open subset of $S \setminus N_i$ can be written as $G \setminus N_i$ where $G \subset S$ is open and $0 \notin G^-$; we need to show that $\liminf_{n \to \infty} \mu_n( G \setminus N_i ) \ge \mu( G \setminus N_i)$. Let $A_1, A_2, \ldots$ be a sequence in $\mathcal{A}$ such that $G = \bigcup_{j \ge 1} A_j$. Write $A_{j,i} = A_j \setminus N_i \in \mathcal{A}$. Since $\mathcal{A}$ is a $\pi$-system and by the condition that $\lim_{n \to \infty} \mu_n(A) = \mu(A)$ for every $A \in \mathcal{A}$, we find, in view of the inclusion-exclusion formula, $\lim_{n \to \infty} \mu_n(\bigcup_{j=1}^k A_{j,i}) = \mu( \bigcup_{j=1}^k A_{j,i} )$ for every integer $k \ge 1$. Let $\varepsilon > 0$. Since $G \setminus N_i = \bigcup_{j \ge 1} A_{j,i}$ and since $\mu(G \setminus N_i) < \infty$, we can find $k$ large enough such that $\mu(G \setminus N_i) \le \mu( \bigcup_{j = 1}^k A_{j,i} ) + \varepsilon$. But $\mu( \bigcup_{j=1}^k A_{j,i} ) = \lim_{n \to \infty} \mu_n( \bigcup_{j=1}^k A_{j,i} )$ is bounded by $\liminf_{n \to \infty} \mu_n( G \setminus N_i )$, as required.
\end{proof}

\begin{acknowledgements}
We thank two reviewers for constructive comments on an earlier version of this paper, pointing out historic references and suggesting various ways to shorten and clarify the paper. In particular, one referee suggested the equivalence of joint regular variation of a time series (i.e., regular variation via finite stretches) with regular variation of the series as a random object in a sequence space. This suggestion eventually led to Theorem~\ref{thm:main}.
\end{acknowledgements}

\bibliographystyle{spbasic}      
\bibliography{B-RV}   

\begin{thebibliography}{31}
\providecommand{\natexlab}[1]{#1}
\providecommand{\url}[1]{{#1}}
\providecommand{\urlprefix}{URL }
\expandafter\ifx\csname urlstyle\endcsname\relax
  \providecommand{\doi}[1]{DOI~\discretionary{}{}{}#1}\else
  \providecommand{\doi}{DOI~\discretionary{}{}{}\begingroup
  \urlstyle{rm}\Url}\fi
\providecommand{\eprint}[2][]{\url{#2}}

\bibitem[{Basrak and Segers(2009)}]{BS09}
Basrak B, Segers J (2009) Regularly varying multivariate time series.
  Stochastic Process Appl 119(4):1055--1080, \doi{10.1016/j.spa.2008.05.004},
  \urlprefix\url{http://dx.doi.org.proxy.bib.ucl.ac.be:8888/10.1016/j.spa.2008.05.004}

\bibitem[{Basrak et~al(2012)Basrak, Krizmani{\'c}, and Segers}]{BKS12}
Basrak B, Krizmani{\'c} D, Segers J (2012) A functional limit theorem for
  dependent sequences with infinite variance stable limits. Ann Probab
  40(5):2008--2033, \doi{10.1214/11-AOP669},
  \urlprefix\url{http://dx.doi.org.proxy.bib.ucl.ac.be:8888/10.1214/11-AOP669}

\bibitem[{Billingsley(1995)}]{BIL95}
Billingsley P (1995) Probability and measure, 3rd edn. Wiley Series in
  Probability and Mathematical Statistics, John Wiley \& Sons Inc., New York, a
  Wiley-Interscience Publication

\bibitem[{Billingsley(1999)}]{Billingsley:1999uc}
Billingsley P (1999) Convergence of probability measures, 2nd edn. Wiley Series
  in Probability and Statistics: Probability and Statistics, John Wiley \& Sons
  Inc.

\bibitem[{Bingham et~al(1987)Bingham, Goldie, and Teugels}]{BIN87}
Bingham NH, Goldie CM, Teugels JL (1987) Regular variation, Encyclopedia of
  Mathematics and its Applications, vol~27. Cambridge University Press,
  Cambridge, \doi{10.1017/CBO9780511721434},
  \urlprefix\url{http://dx.doi.org.proxy.bib.ucl.ac.be:8888/10.1017/CBO9780511721434}

\bibitem[{Davis and Mikosch(2009)}]{DM09}
Davis RA, Mikosch T (2009) The extremogram: a correlogram for extreme events.
  Bernoulli 15(4):977--1009, \doi{10.3150/09-BEJ213},
  \urlprefix\url{http://dx.doi.org.proxy.bib.ucl.ac.be:8888/10.3150/09-BEJ213}

\bibitem[{Davis et~al(2013)Davis, Mikosch, and Zhao}]{MZ15}
Davis RA, Mikosch T, Zhao Y (2013) Measures of serial extremal dependence and
  their estimation. Stochastic Process Appl 123(7):2575--2602,
  \doi{10.1016/j.spa.2013.03.014},
  \urlprefix\url{http://dx.doi.org/10.1016/j.spa.2013.03.014}

\bibitem[{Dombry and Ribatet(2015)}]{dombry+r:2015}
Dombry C, Ribatet M (2015) Functional regular variations, {P}areto processes
  and peaks over threshold. Stat Interface 8(1):9--17,
  \doi{10.4310/SII.2015.v8.n1.a2},
  \urlprefix\url{http://dx.doi.org.proxy.bib.ucl.ac.be:8888/10.4310/SII.2015.v8.n1.a2}

\bibitem[{Drees et~al(2015)Drees, Segers, and Warcho{\l}}]{DSM15}
Drees H, Segers J, Warcho{\l} M (2015) Statistics for tail processes of
  {M}arkov chains. Extremes 18(3):369--402, \doi{10.1007/s10687-015-0217-1},
  \urlprefix\url{http://dx.doi.org/10.1007/s10687-015-0217-1}

\bibitem[{Gin{\'e} et~al(1990)Gin{\'e}, Hahn, and Vatan}]{gine+h+v:1990}
Gin{\'e} E, Hahn MG, Vatan P (1990) Max-infinitely divisible and max-stable
  sample continuous processes. Probab Theory Related Fields 87(2):139--165,
  \doi{10.1007/BF01198427},
  \urlprefix\url{http://dx.doi.org.proxy.bib.ucl.ac.be:8888/10.1007/BF01198427}

\bibitem[{de~Haan and Lin(2001)}]{DHL01}
de~Haan L, Lin T (2001) On convergence toward an extreme value distribution in
  {$C[0,1]$}. Ann Probab 29(1):467--483

\bibitem[{Hult and Lindskog(2005)}]{HL05}
Hult H, Lindskog F (2005) Extremal behavior of regularly varying stochastic
  processes. Stochastic Process Appl 115(2):249--274

\bibitem[{Hult and Lindskog(2006)}]{HL06}
Hult H, Lindskog F (2006) Regular variation for measures on metric spaces. Publ
  Inst Math 80(94):121--140

\bibitem[{Janssen and Drees(2016)}]{JD16}
Janssen A, Drees H (2016) A stochastic volatility model with flexible extremal
  dependence structure. Bernoulli 22(3):1448--1490, \doi{10.3150/15-BEJ699},
  \urlprefix\url{http://dx.doi.org/10.3150/15-BEJ699}

\bibitem[{Janssen and Segers(2014)}]{Janssen+S:2014}
Janssen A, Segers J (2014) Markov tail chains. J Appl Probab 51(4):1133--1153,
  \doi{10.1239/jap/1421763332},
  \urlprefix\url{http://dx.doi.org.proxy.bib.ucl.ac.be:8888/10.1239/jap/1421763332}

\bibitem[{Kuelbs and Mandrekar(1974)}]{kuelbs+m:1974}
Kuelbs J, Mandrekar V (1974) Domains of attraction of stable measures on a
  {H}ilbert space. Studia Math 50:149--162

\bibitem[{Kulik and Soulier(2015)}]{kulik+s:2015}
Kulik R, Soulier P (2015) Heavy tailed time series with extremal independence.
  Extremes 18(2):273--299, \doi{10.1007/s10687-014-0213-x},
  \urlprefix\url{http://dx.doi.org.proxy.bib.ucl.ac.be:8888/10.1007/s10687-014-0213-x}

\bibitem[{Leadbetter(1983)}]{LB83}
Leadbetter MR (1983) Extremes and local dependence in stationary sequences. Z
  Wahrsch Verw Gebiete 65(2):291--306

\bibitem[{Lindskog et~al(2014)Lindskog, Resnick, and Roy}]{LRR14}
Lindskog F, Resnick S, Roy J (2014) {Regularly varying measures on metric
  spaces: Hidden regular variation and hidden jumps}. Probability Surveys
  11:270--314

\bibitem[{Mandrekar and Zinn(1980)}]{mandrekar+z:1980}
Mandrekar V, Zinn J (1980) Central limit problem for symmetric case:
  convergence to non-{G}aussian laws. Studia Math 67(3):279--296

\bibitem[{Meinguet(2012)}]{meinguet:2012}
Meinguet T (2012) Maxima of moving maxima of continuous functions. Extremes
  15(3):267--297, \doi{10.1007/s10687-011-0136-8},
  \urlprefix\url{http://dx.doi.org.proxy.bib.ucl.ac.be:8888/10.1007/s10687-011-0136-8}

\bibitem[{Meinguet and Segers(2010)}]{SER10}
Meinguet T, Segers J (2010) {Regularly varying time series in Banach spaces}.
  arXiv:10013262 [mathPR] \urlprefix\url{http://arxiv.org/abs/1001.3262}

\bibitem[{Mikosch and Wintenberger(2014)}]{mikosch+w:2014}
Mikosch T, Wintenberger O (2014) The cluster index of regularly varying
  sequences with applications to limit theory for functions of multivariate
  {M}arkov chains. Probab Theory Related Fields 159(1-2):157--196,
  \doi{10.1007/s00440-013-0504-1},
  \urlprefix\url{http://dx.doi.org.proxy.bib.ucl.ac.be:8888/10.1007/s00440-013-0504-1}

\bibitem[{Molchanov(2005)}]{molchanov:2005}
Molchanov I (2005) Theory of random sets. Probability and its Applications (New
  York), Springer-Verlag London, Ltd., London

\bibitem[{Pollard(2002)}]{POL02}
Pollard D (2002) A user's guide to measure theoretic probability, Cambridge
  Series in Statistical and Probabilistic Mathematics, vol~8. Cambridge
  University Press

\bibitem[{Resnick(2002)}]{RES02}
Resnick S (2002) Hidden regular variation, second order regular variation and
  asymptotic independence. Extremes 5(4):303--336 (2003),
  \doi{10.1023/A:1025148622954},
  \urlprefix\url{http://dx.doi.org/10.1023/A:1025148622954}

\bibitem[{Resnick(1986)}]{RES86}
Resnick SI (1986) Point processes, regular variation and weak convergence. Adv
  in Appl Probab 18(1):66--138, \doi{10.2307/1427239},
  \urlprefix\url{http://dx.doi.org.proxy.bib.ucl.ac.be:8888/10.2307/1427239}

\bibitem[{Resnick(1987)}]{RES87}
Resnick SI (1987) Extreme values, regular variation, and point processes,
  Applied Probability. A Series of the Applied Probability Trust, vol~4.
  Springer-Verlag, New York, \doi{10.1007/978-0-387-75953-1},
  \urlprefix\url{http://dx.doi.org.proxy.bib.ucl.ac.be:8888/10.1007/978-0-387-75953-1}

\bibitem[{Resnick(2007)}]{RES07}
Resnick SI (2007) Heavy-tail phenomena. Springer Series in Operations Research
  and Financial Engineering, Springer, New York, probabilistic and statistical
  modeling

\bibitem[{Samorodnitsky and Owada(2012)}]{Owada:2013vu}
Samorodnitsky G, Owada T (2012) Tail measures of stochastic processes or random
  fields with regularly varying tails. Tech. rep., Cornell University, Ithaca,
  NY.,
  https://7c4299fd-a-62cb3a1a-s-sites.googlegroups.com/site/takashiowada54/files/tail.measure0103.pdf

\bibitem[{van~der Vaart and Wellner(1996)}]{VDV96}
van~der Vaart AW, Wellner JA (1996) Weak convergence and empirical processes.
  Springer Series in Statistics, Springer-Verlag

\end{thebibliography}


\end{document}